\documentclass[letterpaper,11pt]{article}
\usepackage{color}

\usepackage[small]{titlesec}
\usepackage{theorem,amsmath,amssymb,amscd,bbm}
\usepackage[all,cmtip]{xy}
\usepackage{booktabs}
\usepackage[hyperfootnotes=false, colorlinks, linkcolor={blue}, citecolor={magenta}, filecolor={blue}, urlcolor={blue}]{hyperref}
\usepackage[overload]{textcase} % Suppresses automatic capitalization of math code in page headers
\usepackage{comment}

\theoremstyle{change}
\allowdisplaybreaks
\newcommand{\A}{{\mathbb A}}
\newcommand{\Q}{{\mathbb Q}}
\newcommand{\Z}{{\mathbb Z}}
\newcommand{\R}{{\mathbb R}}
\newcommand{\C}{{\mathbb C}}
\newcommand{\F}{{\mathbb F}}

\newcommand{\ee}{\mathrm{e}}
\newcommand{\ii}{\mathrm{i}}
\newcommand{\p}{\mathfrak p}

\newcommand{\GL}{{\rm GL}}

\newcommand{\SL}{{\rm SL}}

\newcommand{\GSp}{{\rm GSp}}

\newcommand{\Sp}{{\rm Sp}}

\newcommand{\sgn}{{\rm sgn}}

\newcommand{\trace}{{\rm tr}}

\newcommand{\tr}{{\rm tr}}
\newcommand{\HH}{\mathbb{H}}

\newcommand{\mat}[4]{\begin{bsmallmatrix}#1&#2\\#3&#4\end{bsmallmatrix}}
\newcommand{\smallmat}[4]{\begin{bsmallmatrix}#1&#2\\#3&#4\end{bsmallmatrix}}
\newcommand{\qed}{\hspace*{\fill}\rule{1ex}{1ex}}
\newcommand{\forget}[1]{}
\def\qdots{\mathinner{\mkern1mu\raise0pt\vbox{\kern7pt\hbox{.}}\mkern2mu
\raise3.4pt\hbox{.}\mkern2mu\raise7pt\hbox{.}\mkern1mu}}

\newenvironment{proof}{\vspace{1ex}\noindent\emph{Proof.}\hspace{0.5em}}
	{\hfill\qed\vspace{2ex}}
\newenvironment{bsmallmatrix}{\left[\begin{smallmatrix}}{\end{smallmatrix}\right]}

\let\originalleft\left
\let\originalright\right
\renewcommand{\left}{\mathopen{}\mathclose\bgroup\originalleft}
\renewcommand{\right}{\aftergroup\egroup\originalright}

\newtheorem{lemma}{Lemma.}[subsection]
\newtheorem{theorem}[lemma]{Theorem.}
\newtheorem{corollary}[lemma]{Corollary.}
\newtheorem{proposition}[lemma]{Proposition.}

\usepackage{tikz}
\usetikzlibrary{shapes.geometric, arrows}

\tikzstyle{globaldark} = [rectangle, rounded corners,
minimum width=3cm, 
minimum height=1cm, 
text centered, 
text width=3cm, 
draw=black, 
fill=cyan!60]

\tikzstyle{ichar} = [rectangle, rounded corners,
minimum width=2.5cm, 
minimum height=1cm, 
text centered, 
text width=2.5cm, 
draw=black, 
fill=cyan!60]

\tikzstyle{globallight} = [rectangle, rounded corners,
minimum width=3cm, 
minimum height=1cm, 
text centered, 
text width=3cm, 
draw=black, 
fill=cyan!30]

\tikzstyle{localdark} = [rectangle, rounded corners,
minimum width=3cm, 
minimum height=1cm, 
text centered, 
text width=3cm, 
draw=black, 
fill=purple!45]

\tikzstyle{locallightfit} = [rectangle, rounded corners,
minimum width=3cm, 
minimum height=1.5cm, 
text centered, 
text width=3cm, 
draw=black, 
fill=purple!30]

\tikzstyle{locallight} = [rectangle, rounded corners,
minimum width=3cm, 
minimum height=1cm, 
text centered, 
text width=3cm, 
draw=black, 
fill=purple!30]

\tikzstyle{classical} = [rectangle, rounded corners,
minimum width=3cm, 
minimum height=1cm, 
text centered, 
text width=3cm, 
draw=black, 
fill=green!30]

\tikzstyle{dchar} = [rectangle, rounded corners,
minimum width=2.5cm, 
minimum height=1cm, 
text centered, 
text width=2.5cm, 
draw=black, 
fill=green!30]

\tikzstyle{arrow} = [thick,->,>=stealth]

\newcommand{\sbt}{\,\begin{picture}(-2,2)(-1,-4)\circle*{2.5}\end{picture}\ }
\begin{document}

\thispagestyle{empty}
\begin{center}
 {\bf\Large Siegel Eisenstein Series with Paramodular Level}

 \vspace{3ex}
 Erin Pierce, Ralf Schmidt

\vspace{3ex}
\begin{minipage}{80ex}
{\small ABSTRACT:  Starting with a primitive Dirichlet character of conductor $N$, we construct a paramodular Siegel Eisenstein series of level $N^2$ and weight $k\geq4$. We calculate the Fourier expansion of the holomorphic Siegel modular form thus constructed. The function is a paramodular newform, and its adelization generates an irreducible automorphic representation.}
\end{minipage}

\end{center}

\tableofcontents

\section{Introduction}
Let $k,n$ be positive integers with $k$ even. Assuming $k>n(n+1)$, Siegel \cite{Siegel1935} considered the weight $k$ Eisenstein series
\begin{equation}\label{introeq1}
 E_k(Z)=\sum_{(C,D)}\det(CZ+D)^{-k},
\end{equation}
where $Z$ is an element of the Siegel upper half space $\HH_n$ of degree $n$, and $(C,D)$ run over equivalence classes of ``coprime symmetric pairs'' (the second rows of symplectic $2n\times2n$ matrices). Braun \cite{Braun1939} proved that the series converges absolutely for all $Z$ if and only if $k>n+1$. Siegel \cite{Siegel1939} developed a general theory of what we now call Siegel modular forms. The work \cite{Siegel1939} also contains a formula for the Fourier coefficients of $E_k(Z)$, which is sufficiently explicit to see that they are rational numbers. This formula has been made entirely explicit in \cite{Katsurada1999}.

For degree $n=2$ a concrete formula for the Fourier coefficients of $E_k(Z)$ has been known since the work of Maass \cite{Maass1964, Maass1972}. A simplified version of the formula, with a simplified proof, is given in~\cite{EichlerZagier1985} (Corollary~2 to Theorem~6.3). In this paper we will reproduce this formula as a special case of a more general paramodular Siegel Eisenstein series.

The sum \eqref{introeq1} can be rewritten as
\begin{equation}\label{introeq2}
 E_k(Z)=\sum_{\gamma\in P(\Z)\backslash\Gamma}\det(J(\gamma,Z))^{-k},
\end{equation}
where $\Gamma=\Sp(2n,\Z)$, $J(\mat{A}{B}{C}{D},Z)=CZ+D$, and $P(\Z)$ is the subgroup of $\Gamma$ consisting of elements with $C=0$. Suppose that $K$ is a congruence subgroup of $\Sp(2n,\Q)$, by which we mean a group of the form $\Sp(2n,\Q)\cap\prod_{p<\infty}K_p$, where $K_p$ is an open-compact subgroup of $\Sp(2n,\Q_p)$, and $K_p=\Sp(2n,\Z_p)$ for almost all $p$. One may attempt to define Eisenstein series with respect to $K$ by
\begin{equation}\label{introeq3}
 E_k(Z,K,\delta)=\sum_{\gamma\in(K\cap\delta^{-1}P(\Q)\delta)\backslash K}\det(J(\delta\gamma,Z))^{-k},
\end{equation}
where $\delta$ is a fixed element of $\Sp(2n,\Q)$ (so that $E_k(Z,\Gamma,1)=E_k(Z)$). The summation over cosets is well-defined, since a compactness argument shows that any $\mat{*}{*}{}{A}\in\delta K\delta^{-1}\cap P(\Q)$ satisfies $\det(A)=\pm1$. Up to a scalar, $E_{k}(Z,K,\delta)$ depends only on the double coset $P(\Q)\delta K$. The space $P(\Q)\backslash\Sp(2n,\Q)/K$ parametrizes the $0$-dimensional cusps of the Satake compactification of~$\HH_n/K$ (see the exposition in \cite{PoorYuen2013}). Hence the number of $0$-dimensional cusps is an upper bound for the space spanned by the functions \eqref{introeq3}.

In this work we consider $n=2$ and the paramodular group of level $M$,
\begin{equation}\label{introeq4}
 K(M)=\Sp(4,\Q)\cap\begin{bsmallmatrix}\Z&M\Z&\Z&\Z\\\Z&\Z&\Z&M^{-1}Z\\\Z&M\Z&\Z&\Z\\M\Z&M\Z&M\Z&\Z\end{bsmallmatrix}.
\end{equation}
Our starting point is a primitive Dirichlet character $\eta$ of conductor $N$. Let
\begin{equation}\label{introeq5}
 E_{k,\eta}(Z)=\frac12\sum_{b\in(\Z/N\Z)^\times}\eta(b)E_k(Z,K(N^2),\mathcal{C}_0(bN)),\qquad\text{where }\mathcal{C}_0(x)=\begin{bsmallmatrix}1\\&1\\&x&1\\x&&&1\end{bsmallmatrix}.
\end{equation}
By Theorem~1.3 of \cite{PoorYuen2013}, quoted here as Proposition~\ref{zerodimensionalcuspsprop}, the $\mathcal{C}_0(bN)$ represent some, but in general not all, of the $0$-cusps for the group~$K(N^2)$. Clearly $E_{k,\eta}(Z)$ is a holomorphic Siegel modular form of weight~$k$ with respect to~$K(N^2)$. We will show that the adelization of $E_{k,\eta}(Z)$ generates an irreducible automorphic representation $\pi$ and determine its local components; see Theorem~\ref{maintheorem}. We will also show that $\pi$ does not contain any paramodular vectors of lower level. Consequently $E_{k,\eta}(Z)$ is a newform in the sense that any tracing operation to a group $K(M)$ with $M$ a proper divisor of $N^2$ must be zero; see Theorem~\ref{maintheorem}.

If one has the goal of constructing a well-behaved paramodular Eisenstein series associated to a Dirichlet character of conductor~$N$, then it is of course not obvious that one should be aiming for paramodular level $N^2$. We are naturally lead to the function \eqref{introeq5} by adelizing $\eta$ to an idele class character $\chi=\otimes\chi_p$, using it to construct a global induced representation, factoring it into local representations, choosing a global ``section'' $f_s$ by piecing together appropriate local sections $f_{s,p}$, and defining the resulting adelic Eisenstein series $\mathbf{E}_\chi(g,s,f)$; see Sects.~\ref{indrepsec} and~\ref{Eisdefsec}. The local sections at finite places $p$ are chosen to be paramodular newforms, which according to the theory of \cite{NF} have paramodular level $p^{2v_p(N)}$, hence explaining the resulting global level~$N^2$. The local section at the archimedean place is the essentially unique weight-$k$ vector. Once the adelic Eisenstein series is defined, the descent procedure explained in Sect.~\ref{descentsec} results in the linear combination~\eqref{introeq5}.

We then proceed to determine the Fourier expansion of $E_{k,\eta}(Z)$, which is also done in an adelic way. The adelic Fourier coefficients $c_T$, defined in \eqref{cPdefeq}, are functions on $\GSp(4,\A)$, indexed by semi-integral matrices $T=\mat{n}{r/2}{r/2}{m}$. They decompose into four pieces $c_{T,1}, \ldots c_{T,4}$, each of which unfolds in certain ways into a product of local integrals. In this way the Fourier coefficient calculation comes down to the evaluation of local integrals, both archimedean and non-archimedean. The most difficult calculation is that for $c_{T,4}$ if $\mathrm{rank}(T)=2$. Sect.~\ref{unramcalcsec} is dedicated to this case if $p\nmid N$, and Sect.~\ref{ramcalcsec} treats the case $p\mid N$.

We find that the archimedean integrals, and hence the Fourier coefficients $a(T)$, are zero unless $T$ is positive semidefinite. This of course is an instance of the Koecher principle.

Once the adelic Fourier coefficents are determined, the descent procedure immediately gives the classical Fourier coefficents of $E_{k,\eta}(Z)$. Our final result, Theorem~\ref{Fourierexpansiontheorem}, is phrased without any reference to the adelic method. We note that some Fourier coefficient formulas contain a local quantity $K(s,T,\chi_p)$, defined in \eqref{rank2badeq4}, for which we give a general formula only if $\chi_p$ is quadratic and $p$ is odd. In this case some of the $K(s,T,\chi_p)$ can be expressed in terms of the number of solutions of certain elliptic curves over finite fields; see Proposition~\ref{calcofKr=0lemma}. More cases can be found in \cite{Pierce2025}.

The following diagram summarizes our method. The top row is the representation theory level, the bottom row the classical level. The main point is that the adelic approach provides a natural factorization into local objects (the rightmost column). This is useful for two main reasons. First, the factorization of the global induced representation into local representations guides our choice of local sections. Second, the unfolding procedure leads to adelic integrals that factor into local integrals, essentially reducing the determination of the Fourier coefficients to more tractable local calculations.

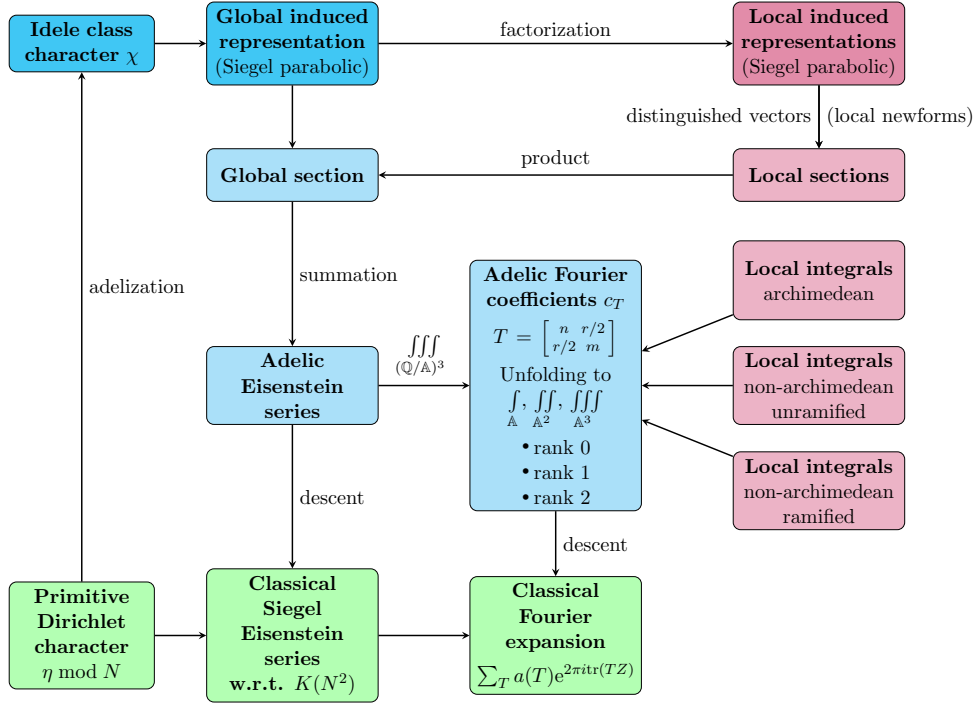
\begin{center}\scalebox{0.7}{\begin{tikzpicture}[scale=0.8, node distance=2cm]

%Global%%%%%%%%%%%%%%%%%%%%%%%%%%%%%%%%%%%%%%%%%%

\node (grep) [globaldark] {\textbf{Global induced representation}\\ (Siegel parabolic)};

\node (gsec) [globallight, below of=grep, yshift=-0.5cm] {\textbf{Global section}};

\node (ichar) [ichar, left of=grep, xshift=-2cm] {\textbf{Idele class character $\chi$}};

\node (aes) [globallight, below of=gsec, yshift=-2cm] {\textbf{Adelic} \\\textbf{Eisenstein series}};

\node (afc) [globallight, right of=aes, xshift=3cm] {\textbf{Adelic Fourier}\\
\textbf{coefficients} $c_T$\\ [1ex]
$T=\begin{bsmallmatrix}
    n&r/2\\r/2&m
\end{bsmallmatrix}$\\ [1ex]
Unfolding to $\int\limits_{\A}$, $\iint\limits_{\A^2}$, $\iiint\limits_{\A^3}$\\ [1ex]
$\sbt \text{ rank }0$\\
$\sbt \text{ rank }1$\\
$\sbt \text{ rank }2$
};

%Local%%%%%%%%%%%%%%%%%%%%%%%%%%%%%%%%%%%%%%%%%%

\node (lrep) [localdark, right of=grep, xshift=8cm] {\textbf{Local induced representations}\\ (Siegel parabolic)};

\node (lsec) [locallight, below of=lrep, yshift=-0.5cm] {\textbf{Local sections}};

\node (linta) [locallightfit, right of=afc, xshift=3cm, yshift=2cm] {\textbf{Local integrals}\\archimedean};

\node (lintu) [locallight, right of=afc, xshift=3cm] {\textbf{Local integrals}\\non-archimedean unramified};

\node (lintr) [locallight, right of=afc, xshift=3cm, yshift=-2cm] {\textbf{Local integrals}\\non-archimedean ramified};

%Classical%%%%%%%%%%%%%%%%%%%%%%%%%%%%%%%%%%%%%%

\node (ces) [classical, below of=aes, yshift=-2.75cm] {\textbf{Classical}\\\textbf{Siegel}\\ \textbf{Eisenstein series}\\ \textbf{w.r.t.} $K(N^2)$};

\node (dchar) [dchar, left of=ces, xshift=-2cm] {\textbf{Primitive Dirichlet character $\eta\bmod N$}};

\node (cfe) [classical, right of=ces, xshift=3cm] {\textbf{Classical Fourier}\\
\textbf{expansion}\\ [1ex] $\sum_T a(T) \ee^{2\pi i\tr(TZ)}$};

\draw [arrow] (lsec) -- node[anchor=south] {product} (gsec);
\draw [arrow] (gsec) -- node[anchor=west] {summation} (aes);
\draw [arrow] (lrep) -- node[anchor=east] {distinguished vectors} (lsec);
\draw [arrow] (lrep) -- node[anchor=west] {(local newforms)\vphantom{g}} (lsec);
\draw [arrow] (grep) -- node[anchor=south] {factorization} (lrep);
\draw [arrow] (grep) -- (gsec);
\draw [arrow] (aes) -- node[anchor=west] {descent} (ces);
\draw [arrow] (aes) -- node[anchor=south] {$\iiint\limits_{(\Q/\A)^3}$} (afc);
\draw [arrow] (ces) -- (cfe);
\draw [arrow] (linta) -- (afc);
\draw [arrow] (lintu) -- (afc);
\draw [arrow] (lintr) -- (afc);
\draw [arrow] (afc) -- node[anchor=west] {descent} (cfe);
\draw [arrow] (ichar) -- (grep);
\draw [arrow] (dchar) -- node[anchor=west, yshift=2em] {adelization} (ichar);
\draw [arrow] (dchar) -- (ces);

\end{tikzpicture}}
\end{center}

The modest goal of this work is to consider the holomorphic Eisenstein series $E_{k,\eta}(Z)$ for $k\geq4$ and calculate its Fourier expansion. While we set up the global induced representation~\eqref{indrepeq2} and the resulting Eisenstein series~\eqref{setupeq3b} with a complex parameter~$s$, we will not consider questions of analytic continuation. We could simply have set $s=k$ from the beginning. However, many of our formulas hold for arbitrary~$s$, for example the unramified local integral in Corollary~\ref{rank2cT4np0cor}. In such cases we prefer to state the results in this greater generality, hoping they might find future applications. The reason that not \emph{all} our formulas are stated for arbitrary $s$ is archimedean, where some integrals cannot be evaluated by elementary functions unless $s=k$; see Proposition~\ref{rank2cT4realprop} as an example.

\textbf{Acknowledgements.} We would like to thank Jackson Morrow and Cris Poor for their helpful comments. 
\section{Adelic and classical Eisenstein series}
\subsection{Notation}
For any commutative ring $R$ with $1$, let $J$ be any non-degenerate anti-symmetric $4\times4$-matrix, and let $\GSp(4,R)$ be the group of $g\in\GL(4,R)$ for which there exists $\lambda(g)\in R^\times$ with $^tgJg=\lambda(g)J$. Those $g$ with $\lambda(g)=1$ form the symplectic group~$\Sp(4,R)$. Two convenient choices for $J$ are
\begin{equation}\label{J1J1eq}
 J_1=\begin{bsmallmatrix}&&&1\\&&1\\&-1\\-1\end{bsmallmatrix}\qquad\text{and}\qquad J_2=\begin{bsmallmatrix}&&1\\&&&1\\-1\\&-1\end{bsmallmatrix}.
\end{equation}
For many purposes, in particular those involving representations, it is more convenient to work with $J_1$, but for the classical theory of Siegel modular forms it is imperative to work with $J_2$. Fortunately there is a simple isomorphism between the two versions of the symplectic group, given (in either direction) by switching the first two rows and columns of a matrix. We denote this operation by $g\mapsto g'$. Our policy in this paper (with the exception of the introduction) will be to define $\GSp(4)$ with respect to $J_1$, and use the $g'$ notation as required.

Let $P(R)$ (resp.~$P^{\scriptscriptstyle{1}}(R)$) be the Siegel parabolic subgroup of $\GSp(4,R)$ (resp.~$\Sp(4,R)$), consisting of all matrices whose lower left $2\times2$ block is zero. The typical element of $P(R)$ can be written as
\begin{equation}\label{PRelementeq}
 \mat{A}{*}{}{u\hat A},\qquad u\in R^\times,\:A\in\GL(2,R),\:\hat A=\mat{}{1}{1}{}\,^t\!A^{-1}\mat{}{1}{1}{}.
\end{equation}
Let $Z(R)$ be the center of~$\GSp(4,R)$, consisting of scalar matrices. We let 
\begin{equation}\label{s1s2defeq}
 s_1=\begin{bsmallmatrix}&1\\1\\&&&1\\&&1\end{bsmallmatrix},\qquad s_2=\begin{bsmallmatrix}1\\&&1\\&-1\\&&&1\end{bsmallmatrix}.
\end{equation}
For local fields, we fix maximal compact subgroups, as follows. For finite primes, let $K_p=\GSp(4,\Z_p)$. Over $\R$, let 
\begin{equation}\label{Kinftyeq}
 K_\infty=\left\{\mat{A}{B}{-B}{A}'\mid A+\ii B\in{\rm U}(2)\right\}.
\end{equation}
Then $K_\infty$ is a maximal compact subgroup of $\Sp(4,\R)$ (and is of index $2$ in a maximal compact subgroup of $\GSp(4,\R)$).

Whenever we work with $p$-adic numbers, we denote by $\p$ the maximal ideal $p\Z_p$ of $\Z_p$. We denote by $K(\p^n)$ the local paramodular group of level $\p^n$ inside $\GSp(4,\Q_p)$, i.e.,
\begin{equation}\label{Kpndefeq}
 K(\p^n)=\{g\in\GSp(4,\Q_p)\mid\lambda(g)\in\Z_p^\times\}\cap\begin{bsmallmatrix}\Z_p&\Z_p&\Z_p&\p^{-n}\\\p^n&\Z_p&\Z_p&\Z_p\\\p^n&\Z_p&\Z_p&\Z_p\\\p^n&\p^n&\p^n&\Z_p\end{bsmallmatrix}.
\end{equation}
The global paramodular group of level $N=\prod p^{n_p}$ is
\begin{align}\label{setupeq3c}
 K(N)=\Sp(4,\Q)\cap\prod_{p<\infty}K(\p^{n_p})=\Sp(4,\Q)\cap\begin{bsmallmatrix}\Z&\Z&\Z&N^{-1}\Z\\N\Z&\Z&\Z&\Z\\N\Z&\Z&\Z&\Z\\N\Z&N\Z&N\Z&\Z\end{bsmallmatrix}.
\end{align}
Similarly as in \cite{PoorYuen2013}, we will use the notation
\begin{equation}\label{C0defeq}
 \mathcal{C}_0(x)=\begin{bsmallmatrix}1\\&1\\x&&1\\&x&&1\end{bsmallmatrix}.
\end{equation}
We will also use the notation
\begin{equation}\label{deltaeq}
 \delta_*=\begin{cases}
           1&\text{if the condition $*$ is satisfied},\\
           0&\text{if the condition $*$ is not satisfied}.
          \end{cases}
\end{equation}  For a nonzero integer $r=\pm \prod p^{\alpha_p}$ and a positive integer $N$, we set
\begin{equation}\label{Nnotation}
    r_{\scriptscriptstyle N}=\prod_{p\mid N} p^{\alpha_p} \quad \text{ and }\quad r_{\scriptscriptstyle {\hat N}}=\pm\prod_{p\nmid N} p^{\alpha_p}.
\end{equation} Clearly, $r=r_{\scriptscriptstyle N} r_{\scriptscriptstyle {\hat N}}$.  For a Dirichlet character $\eta$ and an integer $\ell$, we let
\begin{equation}\label{divisorsumwithcharacter}
    \sigma_{\ell,\eta}(a)=\sum_{d\mid a}\eta^{-1}(d)\, d^\ell,
\end{equation} for any nonzero integer $a$. If $\eta$ is the principal character, we write $\sigma_\ell(a)$ instead of $\sigma_{\ell,\eta}$.

\subsection{Induced representations}\label{indrepsec}
Our starting point is a primitive Dirichlet character $\eta$ of conductor $N$. Let $N=\prod p^{n_p}$ be the prime factorization of $N$. Recall that there is an associated unitary idele class character $\chi:\Q^\times\R_{>0}\backslash\A^\times\to\C^\times$, characterized by the following properties.
\begin{itemize}
 \item If $\chi=\otimes_{p\leq\infty}\chi_p$, where $\chi_p$ is a character of $\Q_p^\times$, then the conductor exponent of $\chi_p$ is $a(\chi_p)=n_p$, for all primes $p$.
 \item For a non-zero integer $d$ relatively prime to $N$,
 \begin{equation}\label{Dirichletcorrespondenceeq4}
     \displaystyle\chi_\infty(d)\prod_{\substack{p<\infty\\p\nmid N}}\chi_p(d)=\prod_{\substack{p<\infty\\p\mid N}}\chi_p(d)^{-1}=\eta(d).
 \end{equation}
 In particular, $\chi_p(p)=\eta(p)$ for primes $p\nmid N$.
 \item If $p\mid N$, then $\chi_p(p)=\eta(x_p)$, where $x_p\in\Z$ is such that
  \begin{equation}\label{Dirichletcorrespondenceeq6}
   x_p\equiv\begin{cases}
           p\bmod q^{n_q}&\text{for }q\mid N,\:q\neq p,\\
           1\bmod p^{n_p}.
          \end{cases}
  \end{equation}
 \item The archimedean component is given by
  \begin{equation}\label{Dirichletcorrespondenceeq7}
   \chi_\infty=\begin{cases}
           1&\text{if }\eta(-1)=1,\\
           \sgn&\text{if }\eta(-1)=-1.
          \end{cases}
  \end{equation}
\end{itemize}
We consider the global induced representation
\begin{equation}\label{indrepeq2}
 J_\chi(s)=\chi|\cdot|_\A^{s-3/2}1_{\GL(2,\A)}\rtimes\chi^{-1}|\cdot|_\A^{-s+3/2}.
\end{equation}
The transformation property of the functions $f$ in the standard model of the representation~\eqref{indrepeq2} is
\begin{equation}\label{indrepeq3}
 f(\mat{A}{*}{}{u\hat A}g)=\chi(u^{-1}\det(A))\,|u^{-1}\det(A)|^{s}f(g)
\end{equation}
for $A\in\GL(2,\A)$, $u\in\A^\times$ and $g\in\GSp(4,\A)$. There is a natural factorization
\begin{equation}\label{indrepeq7}
 J_\chi(s)\cong\bigotimes\limits_p J_{\chi_p}(s),
\end{equation}
with the local representations $J_{\chi_p}(s)$ defined analogously. If $p$ is such that $\chi_p$ is unramified, then $J_{\chi_p}(s)$ contains a unique normalized spherical vector~$f_p^\circ$, characterized by being right invariant under $K_p=\GSp(4,\Z_p)$, and $f_p^\circ(1)=1$. Let $f_p\in J_{\chi_p}(s)$ be given, for each place $p$, such that $f_p=f_p^\circ$ for almost all~$p$. Then we denote by $\otimes f_p$ the function $\GSp(4,\A)\to\C$ given by $(\otimes f_p)(g)=\prod_{p\leq\infty}f_p(g_p)$. Clearly, $\otimes f_p\in J_\chi(s)$.

Let $\hat\delta$ be the function on $P(\A)$ given by
\begin{equation}\label{indrepeq4}
 \hat\delta(\mat{A}{*}{}{u\hat A})=|u^{-1}\det(A)|.
\end{equation}
Using the Iwasawa decomposition, we extend $\hat\delta$ to a function on all of $\GSp(4,\A)$ by making it right $K$-invariant, where $K=\prod_{p\leq\infty}K_p$.

We adopt the following notation. For $f\in J_\chi(s_0)$, let $f_s=\hat\delta^{s-s_0}f$. Then $f_s\in J_\chi(s)$. We call the family $f_s$ the \emph{flat section} containing~$f$.

A function $f$ on $\GSp(4,\A)$ is \emph{$K$-finite}, if the space spanned by the functions $f(\cdot\,\kappa)$, where $\kappa$ runs through~$K$, is finite-dimensional. Similar definitions hold in the local cases. By smoothness, for finite primes, any element of $J_{\chi_p}(s)$ is $K_p$-finite. If $f=\otimes f_p$, then $f$ is $K$-finite if and only if $f_\infty$ is $K_\infty$-finite.
\subsection{Definition of the Eisenstein series}\label{Eisdefsec}
We will now make a particular choice of local sections for each $p$, starting at the archimedean place. The $K_\infty$-types of $\Sp(4,\R)$ are parametrized by pairs of integers $(\ell,\ell')$ with $\ell\geq\ell'$; see \cite{Muic2009} or \cite{Schmidt2017}. The scalar $K_\infty$-types are $(k,k)$, where $k\in\Z$. We denote these also by $\alpha_k$. Explicitly,
\begin{equation}\label{alphakdefeq}
   \alpha_k(\mat{A}{B}{-B}{A}')=\det(A+\ii B)^k
\end{equation}
for $A+\ii B\in{\rm U}(2)$ (see \eqref{Kinftyeq}).

The archimedean component of the global character $\chi$ is $\chi_\infty=\sgn^\epsilon$ with $\epsilon=\frac{1-\eta(-1)}2$. Hence the archimedean component of the representation $J_\chi(s)$ defined in \eqref{indrepeq2} is
\begin{equation}\label{Jarcheq1}
 J_{\chi_\infty}(s)=\sgn^\epsilon|\cdot|_\infty^{s-3/2}1_{\GL(2,\R)}\rtimes\sgn^\epsilon|\cdot|_\infty^{-s+3/2}.
\end{equation}
 It is known (see, for example, Lemma 6.1 of \cite{Muic2009}) that if $\ell,\ell'$ are integers with $\ell\geq\ell'$, then 
\begin{equation}\label{Jarcheq2}
 \text{the multiplicity of the $K_\infty$-type $(\ell,\ell')$ in $J_{\chi_\infty}(s)$ is }\begin{cases}
  1&\text{if }\ell\equiv\ell'\equiv\epsilon\bmod2,\\
  0&\text{otherwise}.
 \end{cases}
\end{equation}
In particular, if $\eta$ is an even (odd) Dirichlet character, then $J_{\chi_\infty}(s)$ contains $\alpha_k$ with $k$ even (odd). In the following we will fix a positive integer $k$ with
\begin{equation}\label{Jarcheq6}
 \chi_\infty(-1)=\eta(-1)=(-1)^k.
\end{equation}
Let $f^{(k)}_{s,\infty}$ be a vector spanning the $K_\infty$-type $(k,k)$ in $J_{\chi_\infty}(s)$, normalized so that $f^{(k)}_{s,\infty}(1)=1$. Explicity, for $A\in\GL(2,\R)$ and $\kappa\in K_\infty$,
\begin{equation}\label{Jarcheq3}
 f^{(k)}_{s,\infty}(\mat{A}{*}{}{u\hat A}\kappa)=\sgn^k(u^{-1}\det(A))\,|u^{-1}\det(A)|_\infty^s\,\alpha_k(\kappa).
\end{equation}
For integers $\ell,k$ with $k\geq\ell>0$, let $\mathcal{D}(k,\ell)$ be the lowest weight representation with minimal $K_\infty$-type $(k,\ell)$; see Sect.~2 of~\cite{Schmidt2017}. By Sects.~9, 10 of~\cite{Muic2009},
\begin{equation}\label{Jarcheq5}
 \mathcal{D}(k,k)\subset J_{\sgn^k}(k).
\end{equation}
Hence, if we specialize to $s=k$, then $f_{s,\infty}^{(k)}$ becomes the lowest weight vector in the subrepresentation $\mathcal{D}(k,k)$ of $J_{\chi_\infty}(k)$.  This ensures that the Siegel modular form resulting from descending to $\HH_2$ is holomorphic.

We turn to the finite places. For any prime $p$ it is known by Theorem~5.2.2 of~\cite{NF} that $J_{\chi_p}(s)$ has minimal paramodular level $n_p=a(\chi_p)$. (For unitary $\chi_p$ and ${\rm Re}(s)\notin\{0,1,2,3\}$, the representation $J_{\chi_p}(s)$ is an irreducible representation of type IIb.) It is also known that the space of $K(\p^{n_p})$-invariant vectors is $1$-dimensional. We let $f_{s,p}$ be a vector spanning this $1$-dimensional space. By Proposition~5.1.2 of~\cite{NF},
\begin{equation}\label{PMieq1}
 \GSp(4,\Q_p)=\bigsqcup_{i=0}^{n_p}P(\Q_p)\mathcal{C}_0(p^i)K(\p^{2n_p}),
\end{equation}
where $\mathcal{C}_0(p^i)$ is defined in \eqref{C0defeq}. By the proof of Theorem~5.2.2 of~\cite{NF}, the function $f_{s,p}\in J_{\chi_p}(s)$ is supported on the double coset $P(\Q_p)\mathcal{C}_0(p^{n_p})K(\p^{2n_p})$. We normalize $f_{s,p}$ in such a way that
\begin{equation}\label{PMieq2}
 f_{s,p}(\mathcal{C}_0(p^{n_p}))=\chi_p(p^{n_p})^{-1}.
\end{equation}
In particular, if $n_p=0$ (i.e., if $\chi_p$ is unramified), then $f_{s,p}$ is the spherical vector $f_{s,p}^\circ$ considered earlier. The global conductor of the representation $J_\chi(s)$ is $\prod p^{2n_p}=N^2$.
We define the global section $f_s^{(k)}\in J_\chi(s)$ by
\begin{equation}\label{globalsectioneq}
 f_{s}^{(k)}:=f_{s,\infty}^{(k)}\otimes\bigotimes_{p<\infty}f_{s,p}.
\end{equation}
Let $f^{(k)}=f_{0}^{(k)}$. With this choice of section, we define the Eisenstein series
\begin{align}\label{setupeq3b}
 \mathbf{E}_\chi(g,s,f^{(k)}):=\sum_{\gamma\in P(\Q)\backslash\GSp(4,\Q)}f_{s}^{(k)}(\gamma g)=\sum_{\gamma\in P^{\scriptscriptstyle{1}}(\Q)\backslash\Sp(4,\Q)}f_s^{(k)}(\gamma g),
\end{align}
whenever the sum converges. The function $\mathbf{E}_\chi(\cdot,s,f^{(k)})$ is right invariant under $\prod_{p<\infty}K(\p^{2n_p})$.
\subsection{Descent to the upper half space}\label{descentsec}
Let $\mathbb{H}_2$ be the Siegel upper half space of degree $2$. Write $Z=X+\ii Y\in\mathbb{H}_2$ as
\begin{equation}\label{H2coordinates}
 Z=\mat{\tau}{z}{z}{\tau'},\qquad X=\mat{x}{u}{u}{x'},\qquad Y=\mat{y}{v}{v}{y'}.
\end{equation}
We act with $\GSp(4,\R)^+$ on $\mathbb{H}_2$ by
\begin{equation}\label{actiononH2eq1}
 gZ=g'\langle Z\rangle,
\end{equation}
where $\smallmat{A}{B}{C}{D}\langle Z\rangle=(AZ+B)(CZ+D)^{-1}$ is the classical action. For $Z=X+\ii Y$ as in \eqref{H2coordinates}, let
\begin{equation}\label{bZeq1}
 b_Z=\begin{bsmallmatrix}1&&u&x\\&1&x'&u\\&&1\\&&&1\end{bsmallmatrix}\begin{bsmallmatrix}a&b\\&d\\&&d^{-1}&-\frac{b}{ad}\\&&&a^{-1}\end{bsmallmatrix},
\end{equation}
where
\begin{equation}\label{bZeq2}
 a=\sqrt{y'-\frac{v^2}{y}},\qquad b=\frac v{\sqrt{y}},\qquad d=\sqrt{y}.
\end{equation}
Then $b_Z$ is the unique upper triangular element in $\Sp(4,\R)$ with positive diagonal elements for which $b_ZI=Z$; here $I=\smallmat{\ii}{}{}{\ii}$.

For $Z\in\mathbb{H}_2$ and $g=\smallmat{A}{B}{C}{D}'\in\GSp(4,\R)^+$, let $J(g,Z)=CZ+D$. For $F:\mathbb{H}_2\to\C$ and $g\in\GSp(4,\R)^+$, let
\begin{equation}\label{slashoperatoreq}
 (F|_k g)(Z)=\lambda(g)^{k}\det(J(g,Z))^{-k}F(gZ).
\end{equation}
This is the scalar-invariant slash operator. For $k\in\Z$ recall the character $\alpha_k$ of $K_\infty$ defined in~\eqref{alphakdefeq}. We have
\begin{equation}\label{alphakdefe3}
 \alpha_k(g)=\det(J(g,I))^{-k}\qquad\text{for }g\in K_\infty.
\end{equation}
Suppose $\Phi:\Sp(4,\R)\to\C$ is a function such that
\begin{equation}\label{descenteq1}
 \Phi(g\kappa)=\alpha_k(\kappa)\Phi(g)\qquad\text{for all }\kappa\in K_\infty.
\end{equation}
Define a function $F:\mathbb{H}_2\to\C$ by
\begin{equation}\label{descenteq2}
 F(Z)=\det(J(g,I))^k\,\Phi(g),\qquad\text{where $g\in\Sp(4,\R)$ is such that }gI=Z.
\end{equation}
Using \eqref{alphakdefe3}, it is easy to see that $F$ is well-defined. From \eqref{bZeq1}, \eqref{bZeq2} we see
\begin{equation}\label{detJbZIformulaeq}
 \det(J(b_Z,I))=\det(Y)^{-1/2},
\end{equation}
so that
\begin{equation}\label{descenteq3}
 F(Z)=\det(Y)^{-k/2}\,\Phi(b_Z).
\end{equation}
In the next section we will apply the descent procedure $\Phi\mapsto F$ to the adelic Eisenstein series $\mathbf{E}_\chi(g,s,f^{(k)})$, so as to obtain a classical Eisenstein series.
\subsection{The classical summation}
Recall the notation \eqref{C0defeq}. We quote from Theorem~1.3 of \cite{PoorYuen2013}:
\begin{proposition}\label{zerodimensionalcuspsprop}
 Let $M,f,M_0$ be positive integers with $M=f^2M_0$ and $M_0$ square-free. Then
 \begin{equation}\label{zerodimensionalcuspspropeq1}
  \Sp(4,\Q)=\bigsqcup_{c\mid f}\;\bigsqcup_{b\in(\Z/c\Z)^\times/\{\pm1\}}P^{\scriptscriptstyle{1}}(\Q)\mathcal{C}_0(\hat bc)K(M),
 \end{equation}
 where the first sum runs over the positive divisors of $f$, and $\hat b$ is an integer relatively prime to $M$ with $\hat b\equiv b\bmod c$. (For $c=1$ or $c=2$ we choose $b=1$.) The number of double cosets is $\big\lfloor\frac{f+2}2\big\rfloor$.
\end{proposition}

\begin{lemma}\label{gammabZlemma}
 Let $\gamma\in\Sp(4,\R)$ and $Z\in\HH_2$. Then $\gamma b_Z=b_{\tilde Z}\kappa$, where $\tilde Z=\gamma Z$ and $\kappa\in K_\infty$ is such that
 \begin{equation}\label{gammabZlemmaeq1}
  \det(J(\kappa,I))=\sqrt{\frac{\det(\tilde Y)}{\det(Y)}}\det(J(\gamma,Z)).
 \end{equation}
\end{lemma}
\begin{proof}
The relation $\tilde Z=\gamma Z$ follows from applying $\gamma b_Z=b_{\tilde Z}\kappa$ to $I$. From $J(\gamma b_Z,I)=J(b_{\tilde Z}\kappa,I)$ we get $J(\gamma,Z)J(b_Z,I)=J(b_{\tilde Z},I)J(\kappa,I)$. Now \eqref{gammabZlemmaeq1} follows by taking determinants on both sides, observing \eqref{detJbZIformulaeq}.
\end{proof}

% Recall from \eqref{setupeq3b} that
% \begin{align}\label{rewritesummationeq1}
%  \mathbf{E}_\chi(g,s,f^{(k)})&=\sum_{\gamma\in P(\Q)\backslash\GSp(4,\Q)}f_s^{(k)}(\gamma g)=\sum_{\gamma\in P^{\scriptscriptstyle{1}}(\Q)\backslash\Sp(4,\Q)}f_s^{(k)}(\gamma g)
% \end{align}
% for $g\in\GSp(4,\A)$.
% We will attempt to rewrite the summation, assuming $g\in\Sp(4,\R)$.
% 
Recall that $N=\prod p^{n_p}$ is the conductor of our primitive Dirichlet character~$\eta$, and that our local sections $f_{s,p}$ for finite $p$ are right invariant under $K(\p^{2n_p})$. By Proposition~\ref{zerodimensionalcuspsprop},
\begin{equation}\label{rewritesummationeq3}
 \Sp(4,\Q)=\bigsqcup_{c\mid N}\;\bigsqcup_{b\in(\Z/c\Z)^\times/\{\pm1\}}P^{\scriptscriptstyle{1}}(\Q)\mathcal{C}_0(\hat bc)K(N^2),
\end{equation}
with the number of double cosets being $\big\lfloor\frac{N+2}2\big\rfloor$. Substituting into \eqref{setupeq3b}, we see that
\begin{align}\label{rewritesummationeq4}
 \mathbf{E}_\chi(g,s,f^{(k)})&=\sum_{c\mid N}\;\sum_{b\in(\Z/c\Z)^\times/\{\pm1\}}\;\sum_{\gamma\in (K(N^2)\cap\mathcal{C}_0(\hat bc)^{-1}P^1(\Q)\mathcal{C}_0(\hat bc))\backslash K(N^2)}f_s^{(k)}(\mathcal{C}_0(\hat bc)\gamma g).
\end{align}
We have, for $g\in\Sp(4,\R)$ and $\gamma\in K(N^2)$,
\begin{align}\label{rewritesummationeq5}
 f_s^{(k)}(\mathcal{C}_0(\hat bc)\gamma g)&=f_{s,\infty}^{(k)}(\mathcal{C}_0(\hat bc)\gamma g)\prod_{p<\infty}f_{s,p}(\mathcal{C}_0(\hat bc))\nonumber\\
 &=f_{s,\infty}^{(k)}(\mathcal{C}_0(\hat bc)\gamma g)\prod_{p\mid N}f_{s,p}({\rm diag}(1,1,\hat b,\hat b)\mathcal{C}_0(c))\nonumber\\
 &=f_{s,\infty}^{(k)}(\mathcal{C}_0(\hat bc)\gamma g)\prod_{p\mid N}\chi_p(\hat b)^{-1}f_{s,p}(\mathcal{C}_0(c))\nonumber\\
 &\stackrel{\eqref{Dirichletcorrespondenceeq4}}{=}f_{s,\infty}(\mathcal{C}_0(\hat bc)\gamma g)\,\eta(\hat b)\prod_{p\mid N}f_{s,p}(\mathcal{C}_0(c)).
\end{align}
By \eqref{PMieq2}, the product is $0$ unless $c=N$, in which case we have
\begin{align}\label{rewritesummationeq5b}
 f_{s,p}(\mathcal{C}_0(N))&=f_{s,p}({\rm diag}(1,1,Np^{-n_p},Np^{-n_p})\mathcal{C}_0(p^{n_p}))\stackrel{\eqref{indrepeq3}, \eqref{PMieq2}}{=}\chi_p(1/N).
\end{align}
Hence, since $\prod_{p\mid N}\chi_p(1/N)=1$,
\begin{align}\label{rewritesummationeq6}
 \mathbf{E}_\chi(g,s,f^{(k)})=\sum_{b\in(\Z/N\Z)^\times/\{\pm1\}}\eta(b)\sum_{\gamma\in (K(N^2)\cap\mathcal{C}_0(bN)^{-1}P^{\scriptscriptstyle{1}}(\Q)\mathcal{C}_0(bN))\backslash K(N^2)}f_{s,\infty}^{(k)}(\mathcal{C}_0(bN)\gamma g).
\end{align}
Now assume that $g=b_Z$ for some $Z\in\HH_2$. By Lemma~\ref{gammabZlemma} (with $\mathcal{C}_0(\hat bN)\gamma$ instead of $\gamma$), we have $\mathcal{C}_0(bN)\gamma b_Z=b_{\tilde Z}\kappa$, where $\tilde Z=\mathcal{C}_0(bN)\gamma Z$ and $\kappa\in K_\infty$ is such that
\begin{equation}\label{rewritesummationeq8}
 \det(J(\kappa,I))=\sqrt{\frac{\det(\tilde Y)}{\det(Y)}}\det(J(\mathcal{C}_0(bN)\gamma,Z)).
\end{equation}
Hence
\begin{align}\label{rewritesummationeq9}
 f_{s,\infty}^{(k)}(\mathcal{C}_0(bN)\gamma g)&=f_{s,\infty}^{(k)}(b_{\tilde Z}\kappa)=\det(\tilde Y)^{s/2}\alpha_k(\kappa)\stackrel{\eqref{alphakdefe3}}{=}\det(\tilde Y)^{s/2}\det(J(\kappa,I))^{-k}\nonumber\\
 &\stackrel{\eqref{rewritesummationeq8}}{=}\det(\tilde Y)^{s/2-k/2}\det(Y)^{k/2}\det(J(\mathcal{C}_0(bN)\gamma,Z))^{-k}.
\end{align}
At this point we assume $s=k$. Substituting into \eqref{rewritesummationeq6}, we get
\begin{equation}\label{rewritesummationeq10}
 \mathbf{E}_\chi(g,k,f)=\det(Y)^{k/2}\sum_{b\in(\Z/N\Z)^\times/\{\pm1\}}\eta(b)E_k(Z,K(N^2),\mathcal{C}_0(bN)),
\end{equation}
where
\begin{equation}\label{rewritesummationeq11}
 E_k(Z,K(N^2),\mathcal{C}_0(bN)):=\sum_{\gamma\in (K(N^2)\cap\mathcal{C}_0(bN)^{-1}P^{\scriptscriptstyle{1}}(\Q)\mathcal{C}_0(bN))\backslash K(N^2)}\det(J(\mathcal{C}_0(bN)\gamma,Z))^{-k},
\end{equation} as in \eqref{introeq3}.
Keeping in mind \eqref{descenteq3}, we see that the classical function corresponding to $\mathbf{E}_\chi(\cdot,k,f)$ is
\begin{equation}\label{rewritesummationeq12}
 E_{k,\eta}(Z):=\sum_{b\in(\Z/N\Z)^\times/\{\pm1\}}\eta(b)E_k(Z,K(N^2),\mathcal{C}_0(bN)),
\end{equation}
Note that there is a bijection
\begin{align}\label{rewritesummationeq13}
 (K(N^2)\cap\mathcal{C}_0(-bN)P^{\scriptscriptstyle{1}}(\Q)\mathcal{C}_0(bN))\backslash K(N^2)\stackrel{\sim}{\longrightarrow}(K(N^2)\cap\mathcal{C}_0(bN)P^{\scriptscriptstyle{1}}(\Q)\mathcal{C}_0(-bN))\backslash K(N^2),
\end{align}
induced by the map $g\mapsto{\rm diag}(-1,1,1,-1)g$. It is then straightforward to verify that
\begin{align}\label{rewritesummationeq15}
 E_k(Z,K(N^2),\mathcal{C}_0(bN))(Z,k)&=(-1)^kE_k(Z,K(N^2),\mathcal{C}_0(-bN)).
\end{align}
Since $\eta(-1)=(-1)^k$, this confirms that the summation over cosets in \eqref{rewritesummationeq12} is well-defined, and that
\begin{equation}\label{rewritesummationeq16}
 E_{k,\eta}(Z)=\frac12\sum_{b\in(\Z/N\Z)^\times}\eta(b)E_k(Z,K(N^2),\mathcal{C}_0(bN)).
\end{equation}
Let
\begin{align}\label{intersectioneq2}
 H(N)&:={\rm diag}(N,1,1,N^{-1})K(N^2){\rm diag}(N^{-1},1,1,N)=\Sp(4,\Z)\cap\begin{bsmallmatrix}\Z&N\Z&N\Z&\Z\\N\Z&\Z&\Z&N\Z\\N\Z&\Z&\Z&N\Z\\\Z&N\Z&N\Z&\Z\end{bsmallmatrix}.
\end{align}
Conjugation by ${\rm diag}(N,1,1,N^{-1})$ maps the intersection $K(N^2)\cap\mathcal{C}_0(bN)^{-1}P^{\scriptscriptstyle{1}}(\Q)\mathcal{C}_0(bN))$ appearing in \eqref{rewritesummationeq11} onto
\begin{equation}\label{intersectioneq3}
 H(N)\cap\mathcal{C}_0(-b)P^{\scriptscriptstyle{1}}(\Q)\mathcal{C}_0(b)=H(N)\cap\mathcal{C}_0(-b)P^{\scriptscriptstyle{1}}(\Z)\mathcal{C}_0(b).
\end{equation}
It is an exercise to show that this intersection consists of all elements
 \begin{equation}\label{intersectionlemmaeq1}
  \mathcal{C}_0(-b)\begin{bsmallmatrix}1&&xN&yN\\&1&zN&xN\\&&1\\&&&1\end{bsmallmatrix}\begin{bsmallmatrix}\alpha&\beta&&-\beta b'\\\gamma&\delta&-\gamma b'\\&&\alpha&-\beta\\&&-\gamma&\delta\end{bsmallmatrix}\mathcal{C}_0(b)
 \end{equation}
with $\mat{\alpha}{\beta}{\gamma}{\delta}\in\SL(2,\Z)$ and $x,y,z\in\Z$; here, $b'$ is a fixed integer with $bb'\equiv1\bmod N$. Using this fact, the summation in \eqref{rewritesummationeq11} can be rewritten in terms of coprime symmetric pairs. Familiar arguments from the classical theory then show that the summation in \eqref{rewritesummationeq11} is absolutely and locally uniformly convergent for $k\geq4$. Hence the functions $E_k(Z,K(N^2),\mathcal{C}_0(bN))$ are holomorphic under this assumption.
\begin{theorem}\label{maintheorem}
 Let $k\geq4$ be an integer, and let $\eta$ be a primitive Dirichlet character of conductor~$N$. Then the function $E_{k,\eta}(Z)$ defined in \eqref{rewritesummationeq12} is a holomorphic Siegel modular form of weight~$k$ with respect to the paramodular group $K(N^2)$. Its adelization generates an irreducible, automorphic representation $\pi\cong\otimes\pi_p$ of $\GSp(4,\A)$, where $\pi_\infty=\mathcal{D}(k,k)$, the holomorphic discrete series representation of $\GSp(4,\R)$ with scalar minimal $K_\infty$-type $(k,k)$, and $\pi_p=\chi_p|\cdot|_p^{k-3/2}1_{\GL(2,\Q_p)}\rtimes\chi_p^{-1}|\cdot|_p^{-k+3/2}$, a representation of type IIb in the classification of \cite{NF}; here $\chi=\otimes\chi_p$ is the idele class character corresponding to $\eta$. Moreover, $E_{k,\eta}(Z)$ is a newform in the sense that any tracing operation to a paramodular form of level a proper divisor of $N^2$ gives zero.
\end{theorem}
\begin{proof}
We already proved that each $E_{k}(Z,K(N^2),\mathcal{C}_0(bN))$, and hence $E_{k,\eta}(Z)$, is a holomorphic function on $\HH_2$. It is immediate from \eqref{rewritesummationeq11} that each $E_{k}(Z,K(N^2),\mathcal{C}_0(bN))$, and hence $E_{k,\eta}(Z)$, has the transformation property of a paramodular form of level $N^2$. By its origin, the adelization of $E_{k,\eta}(Z)$ is the Eisenstein series $\mathbf{E}_\chi(g,s,f^{(k)})$. Consider the map $\varphi$ from the global induced representation $J_\chi(k)$ into the space of automorphic forms on $\GSp(4,\A)$ given by $f\mapsto \mathbf{E}_\chi(\cdot,k,f)$. Clearly $\varphi$ commutes with right translation. The subrepresentation $V:=\mathcal{D}(k,k)\otimes\bigotimes_{p<\infty} J_{\chi_p}(k)$ of $J_\chi(k)$ is irreducible, because each local component is irreducible. Since $f^{(k)}_k$ lies in $V$, the restriction of $\varphi$ to $V$ maps $V$ isomorphically onto an irreducible subspace of automorphic forms containing $\mathbf{E}_\chi(g,s,f^{(k)})$. This proves the statement about $\pi$. The newform property follows from the fact that the minimal paramodular level of $\pi_p=J_{\chi_p(k)}$ is $p^{2n_p}$.
\end{proof}

\section{Fourier coefficients}
We now proceed to calculate the Fourier expansion of the function $E_{k,\eta}(Z)$.  Our main result in this section is Theorem \ref{Fourierexpansiontheorem}.
\subsection{Fourier expansion}\label{Fourierexpansionsec}
Let $\psi=\prod\psi_p$ be the character of $\Q\backslash\A$ for which $\psi_\infty(x)=\ee^{2\pi ix}$. Let $n,r,m\in\Q$ and $T=\smallmat{n}{r/2}{r/2}{m}$. The $T$-th Fourier coefficient of the Eisenstein series defined in \eqref{setupeq3b}, for any global section $f$, is
\begin{align}\label{cPdefeq}
 c_{T}(g,s,f)=\int\limits_{{\rm Sym}_2(\Q)\backslash{\rm Sym}_2(\A)}\mathbf{E}_\chi(\mat{1}{X}{}{1}'g,s,f)\psi\left(\trace(TX)\right)^{-1}\,dX.
\end{align}
We will also write $c_{n,r,m}$ instead of $c_T$ if convenient. Clearly,
\begin{equation}\label{cPtrafoeeq}
 c_T(\mat{1}{X}{}{1}'g,s,f)=\psi(\trace(TX))c_T(g,s,f)
\end{equation}
for all $X\in{\rm Sym}_2(\A)$ and $g\in\GSp(4,\A)$. A calculation shows that
\begin{equation}\label{cPTAeq}
 c_{T}(\mat{\lambda\hat A}{}{}{A}g,s,f)=c_{\chi,\lambda A^{-1}T\,^t\!A^{-1}}(g,s,f)
\end{equation}
for $A\in\GL(2,\Q)$ and $\lambda\in\Q^\times$. By Fourier inversion,
\begin{equation}\label{Fourierexpansioneq}
 \mathbf{E}_\chi(g,s,f)=\sum_{n,r,m\in\Q}c_{n,r,m}(g,s,f).
\end{equation}
Our goal will be to evaluate $\mathbf{E}_\chi(g,s,f^{(k)})$ at $g\in\Sp(4,\R)$. By strong approximation, these values determine $\mathbf{E}_\chi(\cdot,s,f^{(k)})$. Hence we will evaluate every $c_{n,r,m}(g,s,f^{(k)})$ at $g\in\Sp(4,\R)$. By the right transformation property~\eqref{Jarcheq3} it suffices to evaluate these functions for $g$ being upper triangular with positive diagonal elements. In fact, by \eqref{cPtrafoeeq}, it suffices to evaluate these functions at
\begin{equation}\label{bZeq1b}
 g=\begin{bsmallmatrix}a&b\\&d\\&&d^{-1}&-\frac{b}{ad}\\&&&a^{-1}\end{bsmallmatrix},
\end{equation}
with $a,b,d$ as in \eqref{bZeq2}. For any $g\in\Sp(4,\R)$, suppose that $\lambda,\mu\in\hat\Z$ and $\kappa\in N^{-2}\hat\Z$. Then, by~\eqref{cPtrafoeeq}, either $\psi(n\lambda+r\mu+m\kappa)=1$ or $c_{n,r,m}(g,s,f^{(k)})=0$. Hence, if $c_{n,r,m}(g,s,f^{(k)})\neq0$, then $n,r\in\Z$ and $m\in N^2Z$. The Fourier expansion \eqref{Fourierexpansioneq} therefore reads
\begin{equation}\label{Fourierexpansioneq2}
 \mathbf{E}_\chi(g,s,f^{(k)})=\sum_{\substack{n,r\in\Z\\m\in N^2\Z}}c_{n,r,m}(g,s,f^{(k)})
\end{equation}
for $g\in\Sp(4,\R)$.

Using the Bruhat decomposition, we have
\begin{align}\label{setupeq4}
 \GSp(4,F)&=P(F)\sqcup P(F)s_2\begin{bsmallmatrix}1&\\&1&*\\&&1&\\&&&1\end{bsmallmatrix}\sqcup P(F)s_2s_1\begin{bsmallmatrix}1&*&&*\\&1\\&&1&*\\&&&1\end{bsmallmatrix}\sqcup P(F)s_2s_1s_2\begin{bsmallmatrix}1&&*&*\\&1&*&*\\&&1\\&&&1\end{bsmallmatrix}
\end{align}
for any field $F$. Hence
\begin{align}\label{setupeq5}
 \mathbf{E}_\chi(g,s,f^{(k)})&=f_s^{(k)}(g)+\sum_{\beta\in\Q}f_s^{(k)}(s_2\begin{bsmallmatrix}1&\\&1&\beta\\&&1&\\&&&1\end{bsmallmatrix}g)\nonumber\\
 &+\sum_{\rho,\delta\in\Q}f_s^{(k)}(s_2s_1\begin{bsmallmatrix}1&\rho&&\delta\\&1&&\\&&1&-\rho\\&&&1\end{bsmallmatrix}g)+\sum_{\beta,\gamma,\delta\in\Q}f_s^{(k)}(s_2s_1s_2\begin{bsmallmatrix}1&&\gamma&\delta\\&1&\beta&\gamma\\&&1\\&&&1\end{bsmallmatrix}g).
\end{align}
Substituting into \eqref{cPdefeq}, we see that
\begin{equation}\label{FCcalceq1}
 c_T(g,s,f^{(k)})=\sum_{i=1}^4c_{T,i}(g,s,f^{(k)})
\end{equation}
with
\begin{align}
 \label{FCcalceq2a}c_{T,1}(g,s,f^{(k)})&=\iiint\limits_{(\Q\backslash\A)^3}f_s^{(k)}(\begin{bsmallmatrix}1&&\mu&\kappa\\&1&\lambda&\mu\\&&1\\&&&1\end{bsmallmatrix}g)\psi(n\lambda+r\mu+m\kappa)^{-1}\,d\lambda\,d\mu\,d\kappa,\\
 \label{FCcalceq2b}c_{T,2}(g,s,f^{(k)})&=\iiint\limits_{(\Q\backslash\A)^3}\sum_{\beta\in\Q}f_s^{(k)}(s_2\begin{bsmallmatrix}1\\&1&\beta\\&&1\\&&&1\end{bsmallmatrix}\begin{bsmallmatrix}1&&\mu&\kappa\\&1&\lambda&\mu\\&&1\\&&&1\end{bsmallmatrix}g)\psi(n\lambda+r\mu+m\kappa)^{-1}\,d\lambda\,d\mu\,d\kappa,\\
 \label{FCcalceq2c}c_{T,3}(g,s,f^{(k)})&=\iiint\limits_{(\Q\backslash\A)^3}\sum_{\alpha,\delta\in\Q}f_s^{(k)}(s_2s_1\begin{bsmallmatrix}1&\alpha&&\delta\\&1&&\\&&1&-\alpha\\&&&1\end{bsmallmatrix}\begin{bsmallmatrix}1&&\mu&\kappa\\&1&\lambda&\mu\\&&1\\&&&1\end{bsmallmatrix}g)\psi(n\lambda+r\mu+m\kappa)^{-1}\,d\lambda\,d\mu\,d\kappa,\\
 \label{FCcalceq2d}c_{T,4}(g,s,f^{(k)})&=\iiint\limits_{(\Q\backslash\A)^3}\sum_{\beta,\gamma,\delta\in\Q}f_s^{(k)}(s_2s_1s_2\begin{bsmallmatrix}1&&\gamma&\delta\\&1&\beta&\gamma\\&&1\\&&&1\end{bsmallmatrix}\begin{bsmallmatrix}1&&\mu&\kappa\\&1&\lambda&\mu\\&&1\\&&&1\end{bsmallmatrix}g)\psi(n\lambda+r\mu+m\kappa)^{-1}\,d\lambda\,d\mu\,d\kappa.
\end{align}
Trivially, $c_{T,1}$ calculates to
\begin{equation}\label{FCcalceq3}
 c_{T,1}(g,s,f^{(k)})=\begin{cases}
                       f_s^{(k)}(g)&\text{if }T=0,\\
                       0&\text{if }T\neq0.
                      \end{cases}
\end{equation}
For $g=b_Z$, it follows from \eqref{bZeq2} and \eqref{indrepeq3} that
\begin{equation}\label{FCcalceq4}
 c_{T,1}(b_Z,s,f^{(k)})=\begin{cases}
                       \det(Y)^{\frac s2}&\text{if }T=0\text{ and }N=1,\\
                       0&\text{otherwise}.
                      \end{cases}
\end{equation}
(Recall from Sect.~\ref{Eisdefsec} that $1$ is not in the support of $f_s^{(k)}$ if $N>1$.) The other pieces are more difficult to calculate, but it can be done at least for $s=k\geq4$. For $T=0$ (the constant term) they can be calculated for any $s$, and we give the results below. For $\mathrm{rank}(T)=1$ the general answer involves Whittaker functions; below we only give the result for $s=k$, in which case the Whittaker functions become exponentials. The most difficult case (carried out in the following sections) is for $\mathrm{rank}(T)=2$, where an explicit answer is only possible for $s=k$.

Some of the $c_{T,i}$ are zero for global reasons, some are zero for $p$-adic reasons (if $N>1$), and some are zero for archimedean reasons after we set $s$ equal to $k$. The following table gives the values $c_{T,i}(b_Z,s,f^{(k)})$ (for arbitrary $s$ if $T=0$, for $s=k$ otherwise).
\begin{equation}\label{cTitable}
 \begin{array}{ccccc}
  \toprule
   \text{rank }T&c_{T,1}&c_{T,2}&c_{T,3}&c_{T,4}\\
  \toprule
   0&\eqref{FCcalceq4}&\eqref{cT2rank0eq}&\eqref{cT3rank0eq}&\eqref{cT4rank0eq}\\
   &&\scriptstyle(=0\text{ for }s\to k)&\scriptstyle(=0\text{ for }s\to k)&\scriptstyle(=0\text{ for }s\to k)\\
  \midrule
   1&0&\eqref{cT2rank1eq}&\eqref{cT3rank1eq}&0\\
  \midrule
   2&0&0&0&\eqref{rank2cT4eq46}\\
  \bottomrule
 \end{array}
\end{equation}
The results we state without proofs (which can be found in \cite{Pierce2025}) are as follows. For $T=0$,
\begin{align}
 \label{cT2rank0eq}c_{T,2}(b_Z,s,f^{(k)})&=\begin{cases}
                       \displaystyle\ii^k\,\det(Y)^{\frac s2}2^{2-s}\pi \frac{\Gamma(s-1)}{\Gamma(\frac{s-k}{2})\Gamma(\frac{s+k}{2})}\frac{\zeta(s-1)}{\zeta(s)}y^{1-s}&\text{if }N=1,\\
                       0&\text{if }N>1,
                      \end{cases}\\
 \label{cT3rank0eq}c_{T,3}(b_Z,s,f^{(k)})&=
   \begin{cases}
    \displaystyle\ii^k\det(Y)^{\frac{s}{2}}2^{2-s}\pi\frac{\Gamma(s-1)}{\Gamma(\frac{s+k}{2})\Gamma(\frac{s-k}{2})}\frac{\zeta(s-1)}{\zeta(s)}\left(\frac{\zeta_{Y}(s-1)}{2\zeta(2s-2)}-y^{1-s}\right)&\text{if }N=1,\\[3ex]
    \displaystyle(-\ii)^k\det(Y)^{\frac{s}{2}}2^{1-s}\pi\frac{\Gamma(s-1)}{\Gamma(\frac{s+k}{2})\Gamma(\frac{s-k}{2})}\frac{L(s-1,\eta)\zeta_{Y_N}(s-1,\eta)}{L(s,\eta)L(2s-2,\eta^2)}&\text{if }N>1,
   \end{cases}\\
 \label{cT4rank0eq}c_{T,4}(b_Z,s,f^{(k)})&=\begin{cases}
                       \displaystyle 2^{5-2s}\pi^{\frac{5}{2}}\frac{\det(Y)^{\frac{3-s}{2}}\Gamma(s-\frac{3}{2})\Gamma(s-2)}{\Gamma(\frac{s-k}{2})\Gamma(\frac{s+k}{2})\Gamma(\frac{s-k-1}{2})\Gamma(\frac{s+k-1}{2})}\frac{\zeta(s-2)\zeta(2s-3)}{\zeta(s)\zeta(2s-2)}&\text{if }N=1,\\
                       0&\text{if }N>1.
                      \end{cases}
\end{align}
Here, $Y_N=\mat{y}{vN}{vN}{y'N^2}$ and
\begin{equation}\label{epsteinzetalemma2eq1b}
        \zeta_{Y}(s,\eta):=\sum_{(0,0)\neq(\alpha,\gamma)\in\Z^2}\frac{\eta(\alpha\gamma)}{(y\alpha^2+2v\alpha\gamma+y'\gamma^2)^s}
\end{equation}
is an Epstein zeta function with character. For $\text{rank }T=1$,
\begin{align}
 \label{cT2rank1eq}c_{T,2}(b_Z,k,f^{(k)})&=
  \begin{cases}
   \displaystyle\frac{\det(Y)^{k/2}(2\pi \ii)^k \ee^{2\pi\ii n \tau}\sigma_{k-1}(e)}{(k-1)!\,\zeta(k)}  & \text{if }N=1,\:T=\mat{n}{0}{0}{0},\:n>0,\\
   0&\text{otherwise},
  \end{cases}\\
 \label{cT3rank1eq}c_{T,3}(b_Z,k,f^{(k)})&=
  \begin{cases}
   \displaystyle\frac{\det(Y)^{k/2}(-2\pi\ii)^k \ee^{2\pi\ii(n\tau+rz+m\tau')}\eta(r_{\scriptscriptstyle{\hat N}})\sigma_{k-1,\eta}(e_{\scriptscriptstyle \hat N})}{ (k-1)!\, e_{\scriptscriptstyle N}^{1-k}L(k,\eta) \eta(2_{\scriptscriptstyle{\hat N}})}
   &\text{if }m>0,\:r_{\scriptscriptstyle N}=\frac{(2m)_{\scriptscriptstyle N}}{N},\\
   0&\text{otherwise}.
  \end{cases}
\end{align}

Assume that $T$ has rank $2$. In this case it is easy to see that $c_{T,i}(g,s,f^{(k)})=0$ for $i=1,2,3$. The integral $c_{T,4}(g,s,f^{(k)})$ unfolds to
\begin{align}\label{rank2cT4eq1}
    c_{T,4}(g,s,f^{(k)})&=\int\limits_{\A}\int\limits_{\A}\int\limits_{\A}f_s^{(k)}(s_2s_1s_2\begin{bsmallmatrix}1&&\mu&\kappa\\&1&\lambda&\mu\\&&1\\&&&1\end{bsmallmatrix}g)\psi(n\lambda+r\mu+m\kappa)^{-1}\,d\lambda\,d\mu\,d\kappa=\prod_{p\leq\infty}I_p
\end{align}
with the obvious local integrals $I_p$. For finite $p$ these will be calculated in Sects.~\ref{unramcalcsec} (for $\chi_p$ unramified) and~\ref{ramcalcsec} (for $\chi_p$ ramified). The archimedean integral $I_\infty$ is the topic of the next section.
\subsection{The archimedean local integral}
In this section we evaluate the archimedean local integral
\begin{equation}\label{rank2cT4eq2}
 I_\infty(T,g)=\int\limits_{\R}\int\limits_{\R}\int\limits_{\R} f_{s,\infty}^{(k)}(s_2s_1s_2\begin{bsmallmatrix}1&&\mu&\kappa\\&1&\lambda&\mu\\&&1\\&&&1\end{bsmallmatrix}g)\ee^{-2\pi\ii(n\lambda+r\mu+m\kappa)}\,d\lambda\,d\mu\,d\kappa
\end{equation}
in the factorization \eqref{rank2cT4eq1}. Recall that we may assume $g$ is the element in \eqref{bZeq1b}. It is straightforward to verify that
\begin{equation}\label{rank2cT4realeq3}
 I_\infty(T,\mat{\hat A}{}{}{A}g)=\chi_\infty(\det(A))|\det(A)|^{s-3}I_\infty(A^{-1}T\,^t\!A^{-1},g)
\end{equation}
for any $A\in\GL(2,\R)$. In particular,
\begin{equation}\label{rank2cT4realeq4}
 I_\infty(T,\begin{bsmallmatrix}a&b\\&d\\&&d^{-1}&-\frac{b}{ad}\\&&&a^{-1}\end{bsmallmatrix})=|ad|^{3-s}I_\infty(T',1),\qquad T'=\mat{d^2n+dbr+b^2m}{\frac{adr+2abm}2}{\frac{adr+2abm}2}{a^2m}.
\end{equation}
Hence it is enough to calculate $I_{\infty}(T,1)$.

\begin{proposition}\label{rank2cT4realprop}
 If $s=k\geq2$, then
 \begin{equation}\label{rank2cT4realpropeq1}
  I_\infty(T,1)=
   \begin{cases}
    \displaystyle\frac{(4\pi)^{2k-1}\det(T)^{k-\frac32}\ee^{-2\pi\,{\rm Tr}(T)}}{2(2k-2)!}&\text{if $T$ is positive definite},\\
    0&\text{otherwise}.
   \end{cases}
 \end{equation}
\end{proposition}
\begin{proof}
Since $\chi_\infty={\rm sgn}^k$, it follows from \eqref{alphakdefeq} and \eqref{rank2cT4realeq3} that
\begin{equation}\label{rank2cT4realpropeq2}
 I_\infty(T,1)=I_\infty(AT\,^t\!A,1)\qquad\text{for all }A\in{\rm O}(2).
\end{equation}
Since the right hand side of \eqref{rank2cT4realpropeq1} is also invariant under the transformation $T\mapsto AT\,^t\!A$ with $A\in{\rm O}(2)$, we may assume that $T$ is diagonal, i.e., that $r=0$. Hence we will calculate
\begin{equation}\label{rank2cT4realpropeq3}
 I_\infty(T,1)=\int\limits_\R\int\limits_\R\int\limits_\R f_{s,\infty}^{(k)}(s_2s_1s_2\begin{bsmallmatrix}1&&\mu&\kappa\\&1&\lambda&\mu\\&&1\\&&&1\end{bsmallmatrix})\ee^{-2\pi\ii(n\lambda+m\kappa)}\,d\lambda\,d\mu\,d\kappa.
\end{equation}
Our method is to repeatedly use the Iwasawa decomposition
\begin{equation}\label{SL2NAKeq3}
 \mat{1}{a}{}{1}=\mat{1}{}{x}{1}\mat{y^{-1/2}}{}{}{y^{1/2}}r(\theta),\qquad r(\theta)=\mat{\cos(\theta)}{\sin(\theta)}{-\sin(\theta)}{\cos(\theta)},
\end{equation}
with
\begin{equation}\label{SL2NAKeq4}
 x=\frac a{1+a^2},\qquad y=\frac1{1+a^2},\qquad \ee^{i\theta}=\frac{1+ia}{(1+a^2)^{1/2}},
\end{equation}
first on the $\lambda$-variable, then on $\mu$, then on $\kappa$. The result is the formula
\begin{align}\label{rank2cT4realpropeq6}
 I_\infty(T,1)&=\int\limits_\R\int\limits_\R\int\limits_\R (1+\ii\lambda)^{\frac{1-s+k}2}(1-\ii\lambda)^{\frac{1-s-k}2}\ee^{-2\pi\ii\lambda(n+m\mu^2)}\nonumber\\
 &\qquad (1+\mu^2)^{1-s}(1+\ii\kappa)^{\frac{-s+k}2}(1-\ii\kappa)^{\frac{-s-k}2} \ee^{-2\pi\ii m\kappa(1+\mu^2)}\,d\lambda\,d\mu\,d\kappa.
\end{align}
After setting $s=k$, our assertion follows easily using the integration formula
\begin{equation}\label{intformula5}
 \int\limits_\R(1-\ii x)^{-a}\ee^{-\ii px}\,dx=
 \begin{cases}
   \displaystyle\frac{2\pi p^{a-1}\ee^{-p}}{\Gamma(a)}&\text{if }p>0,\\[2ex]
   \displaystyle0&\text{if }p\leq0.
 \end{cases}
\end{equation}
for $\mathrm{Re}(a)>1$ (see 3.382.7 of \cite{GradRyzh2007}).
\end{proof}

\begin{corollary}\label{rank2cT4realpropcor}
 If $s=k\geq2$, then, with $b_Z$ as in \eqref{bZeq1},
 \begin{equation}\label{rank2cT4realpropcoreq1}
  I_\infty(T,b_Z)=
   \begin{cases}
    \displaystyle\det(Y)^{\frac k2}\frac{(4\pi)^{2k-1}\det(T)^{k-\frac32}\ee^{2\pi\ii\,{\rm Tr}(TZ)}}{2(2k-2)!}&\text{if $T$ is positive definite},\\
    0&\text{otherwise}.
   \end{cases}
 \end{equation}
\end{corollary}
\begin{proof}
This follows from Proposition \ref{rank2cT4realprop}, \eqref{rank2cT4realeq4}, \eqref{cPtrafoeeq} and \eqref{bZeq2}.
\end{proof}

\subsection{The local integral at good places}
The following lemma is well known. Recall that an integer $D$ is a fundamental discriminant if either $D$ is square-free and $D\equiv1\bmod4$, or $D=4D'$ with $D'$ square-free and $D'\equiv2,3\bmod4$.

\begin{lemma}\label{funddisclemma}
 Let $M$ be a non-zero integer with $M\equiv0,1$ mod $4$. Then $M=Df^2$ with a (uniquely determined) fundamental discriminant $D$ and an integer $f$.
\end{lemma}

Recall that $\Delta=4nm-r^2$. We will apply the lemma to $-\Delta=r^2-4nm$. Hence
\begin{equation}\label{DeltaDeq1}
 -\Delta=Df^2
\end{equation}
with a fundamental discriminant $D$ and an integer $f$, which we assume to be positive. Clearly $\Q(\sqrt{-\Delta})=\Q(\sqrt{D})$, and $\Q_p(\sqrt{-\Delta})=\Q_p(\sqrt{D})$ for each prime~$p$. It is an exercise to show that
\begin{equation}\label{DeltaDeq2}
 d_p(\Delta):=v_p\left(\text{discriminant of the extension }\Q_p(\sqrt{-\Delta})/\Q_p\right)=v_p(D).
\end{equation}
Explicitly, $d_p(\Delta)=0$ if the extension is unramified, $d_p(\Delta)=1$ if the extension is ramified and $p$ is odd, $d_p(\Delta)=2$ if $p=2$ and $\Q_2(\sqrt{-\Delta})=\Q_2(\sqrt{u})$ with $u\in\{3,7\}$, and $d_p(\Delta)=3$ if $p=2$ and $\Q_2(\sqrt{-\Delta})=\Q_2(\sqrt{2u})$ with $u\in\{1,3,5,7\}$. Let
\begin{equation}\label{edefeq}
 e=\gcd(n,r,m).
\end{equation}
Clearly $e^2\mid\Delta$. We claim that
\begin{equation}\label{efeq}
 e\mid f.
\end{equation}
Indeed, we have $-\Delta=r^2-4nm=e^2(r'^2-4n'm')$. Applying Lemma~\ref{funddisclemma} to $r'^2-4n'm'$, we get $r'^2-4n'm'=D'f'^2$ with a fundamental discriminant $D'$. Hence $-\Delta=D'(ef')^2$, and we see that $D'=D$ and $ef'=f$.

For later use we write the prime factorizations of $e$ and $f$ as
\begin{equation}\label{efprimefaceq}
 e=\prod p^{e_p},\qquad f=\prod p^{f_p}.
\end{equation}
We have $e_p\leq f_p$ by \eqref{efeq}. 

Let $\chi_D=\big(\frac D\cdot\big)$ be the quadratic Dirichlet character corresponding to the extension $\Q(\sqrt{D})/\Q$. Then
\begin{equation}\label{Lpdefeq}
  \chi_D(p)=\begin{cases}
     1&\text{if $\Q_p(\sqrt{-\Delta})=\Q_p$},\\
     -1&\text{if $\Q_p(\sqrt{-\Delta})/\Q_p$ is unramified quadratic},\\
     0&\text{if $\Q_p(\sqrt{-\Delta})/\Q_p$ is ramified}.
    \end{cases}
\end{equation}
Recall that $n_p=a(\chi_p)$, and that $\eta$ is the primitive Dirichlet character corresponding to $\chi$.

\begin{corollary}\label{rank2cT4np0cor}
 Suppose that $n_p=0$ and $\mathrm{rank}(T)=2$. Then the local integral $I_p$ in \eqref{rank2cT4eq1} is given by
 \begin{align}\label{rank2cT4np0coreq2}
  I_p&=\frac{(1-\eta(p)p^{-s})(1-\eta(p)^2p^{2-2s})}{1-\chi_D(p)\eta(p)p^{1-s}}p^{f_p(3-2s)}\eta(p)^{2f_p}\nonumber\\
  &\sum_{i=0}^{e_p}\eta(p)^{-i}p^{i(s-1)}\Bigg(\sum_{j=0}^{f_p-i}\eta(p)^{-2j}p^{j(2s-3)}-\chi_D(p)\eta(p)^{-1}p^{s-2}\sum_{j=0}^{f_p-i-1}\eta(p)^{-2j}p^{j(2s-3)}\Bigg).
 \end{align}
\end{corollary}
\begin{proof}
Observe that if $s_0$ is such that $\chi_p(p)=p^{-s_0}$, then $J_{\chi_p}(s)=J_1(s+s_0)$. Hence the assertion follows by replacing $s$ by $s+s_0$ in Proposition \ref{rank2goodtheorem}, observing that the quantity $L$ in \eqref{rank2goodtheoremeq1} equals $\chi_D(p)$, that $\chi_p(p)=\eta(p)$, and that $\frac{v_p(\Delta)-d_p(\Delta)}2=f_p$ by \eqref{DeltaDeq1}.
\end{proof}

\subsection{The Fourier coefficients for rank 2}
We will now calculate $c_{T,4}$, given in \eqref{rank2cT4eq1}, assuming $\mathrm{rank} (T)=2$. Recall that $I_\infty$ is given in \eqref{rank2cT4realpropcoreq1} and $I_p$ for finite $p$ is given in \eqref{rank2cT4np0coreq2} (if $p\nmid N$) and \eqref{rank2badeq5} (for $p\mid N$). If $T$ is not positive definite, then $c_{T,4}(b_Z,s,f^{(k)})=0$ by~\eqref{rank2cT4realpropcoreq1}. Assume in the following that $T$ is positive definite. Then
\begin{align}\label{rank2cT4eq31}
  \prod_{p<\infty}I_p=&\prod_{\substack{p<\infty\\n_p=0}}\Bigg(\frac{(1-\eta(p)p^{-s})(1-\eta(p)^2p^{2-2s})}{1-\chi_D(p)\eta(p)p^{1-s}}p^{f_p(3-2s)}\eta(p)^{2f_p}\nonumber\\
    &\hspace{3ex}\sum_{i=0}^{e_p}\eta(p)^{-i}p^{i(s-1)}\Bigg(\sum_{j=0}^{f_p-i}\eta(p)^{-2j}p^{j(2s-3)}-\chi_D(p)\eta(p)^{-1}p^{s-2}\sum_{j=0}^{f_p-i-1}\eta(p)^{-2j}p^{j(2s-3)}\Bigg)\Bigg)\nonumber\\
    &\prod_{\substack{p<\infty\\n_p>0}}\Bigg(\delta_{v_p(m)\geq2n_p}\,p^{n_p(5/2-2s)}\varepsilon(1/2,\chi_p,\psi_p)\nonumber\\
    &\hspace{3ex}\left(\delta_{v_p(r)=0}\chi_p(-r)+\delta_{v_p(r)>0}\,p^{n_p(2-s)}\chi_p(-p^{n_p})K(s,T,\chi_p)\right)\Bigg).
\end{align}
Because of the factor $\delta_{v_p(m)\geq2n_p}$, this is zero unless $N^2\mid m$. In the following we assume that $N^2\mid m$. Let us set, for integers $e,f$ with $e\mid f$,
\begin{equation}\label{tildeHetaeq1}
 \tilde H_{D,s,\eta}(e,f)=\sum_{d\mid e}\eta(d)^{-1}d^{s-1}\sum_{g\mid\frac fd}\mu(g)\chi_D(g)\eta(g)^{-1}g^{s-2}\sum_{h\mid\frac{f}{dg}}\eta(h)^{-2}h^{2s-3}.
\end{equation}
It is easy to confirm that $\tilde H_{D,s,\eta}(e,f)$ is a multiplicative function, i.e.,
\begin{equation}\label{tildeHetaeq2}
 \tilde H_{D,s,\eta}(e_1,f_1)\tilde H_{D,s,\eta}(e_2,f_2)=\tilde H_{D,s,\eta}(e_1e_2,f_1f_2)\qquad\text{if }e_1\mid f_1,\:e_2\mid f_2,\:\gcd(f_1,f_2)=1.
\end{equation}
Recall from \eqref{efprimefaceq} that $e=\prod p^{e_p}$ and $f=\prod p^{f_p}$. Using the notation \eqref{Nnotation}, we have
\begin{align}\label{rank2cT4eq32}
  \prod_{p<\infty}I_p&=N^{\frac52-2s}\prod_{\substack{p<\infty\\n_p=0}}\Bigg(\frac{(1-\eta(p)p^{-s})(1-\eta(p)^2p^{2-2s})}{1-\chi_D(p)\eta(p)p^{1-s}}p^{f_p(3-2s)}\eta(p)^{2f_p}\tilde H_{D,s,\eta}(p^{e_p},p^{f_p})\Bigg)\nonumber\\
    &\qquad\prod_{\substack{p<\infty\\n_p>0}}\Bigg(\varepsilon(1/2,\chi_p,\psi_p)\left(\delta_{v_p(r)=0}\chi_p(-r)+\delta_{v_p(r)>0}\,p^{n_p(2-s)}\chi_p(-p^{n_p})K(s,T,\chi_p)\right)\Bigg)\nonumber\\
  &=N^{\frac52-2s} f_{\scriptscriptstyle\hat N}^{3-2s}\eta( f_{\scriptscriptstyle\hat N}^2)\tilde H_{D,s,\eta}( e_{\scriptscriptstyle\hat N}, f_{\scriptscriptstyle\hat N})\nonumber\\
    &\qquad\prod_{p<\infty}\Bigg(\frac{(1-\eta(p)p^{-s})(1-\eta(p)^2p^{2-2s})}{1-\chi_D(p)\eta(p)p^{1-s}}\Bigg)\Bigg(\prod_{p<\infty}\varepsilon(1/2,\chi_p,\psi_p)\Bigg)\nonumber\\
    &\qquad\prod_{\substack{p<\infty\\n_p>0}}\Bigg(\left(\delta_{v_p(r)=0}\chi_p(-r)+\delta_{v_p(r)>0}\,p^{n_p(2-s)}\chi_p(-p^{n_p})K(s,T,\chi_p)\right)\Bigg)\nonumber\\
%   &=N^{\frac52-2s} f_{\scriptscriptstyle\hat N}^{3-2s}\eta( f_{\scriptscriptstyle\hat N}^2)\tilde H_{D,s,\eta}( e_{\scriptscriptstyle\hat N}, f_{\scriptscriptstyle\hat N})\nonumber\\
%     &\qquad\frac{L(s-1,\chi_D\eta)}{L(s,\eta)L(2s-2,\eta^2)}\Bigg(\prod_{\substack{p<\infty\\n_p>0}}\chi_p(-1)\Bigg)\Bigg(\prod_{p<\infty}\varepsilon(1/2,\chi_p,\psi_p)\Bigg)\nonumber\\
%     &\qquad\Bigg(\prod_{\substack{p<\infty\\p\mid N\\p\nmid r}}\chi_p(r)\Bigg)\Bigg(\prod_{\substack{p<\infty\\p\mid N\\p\mid r}}p^{n_p(2-s)}\chi_p(p^{n_p})K(s,T,\chi_p)\Bigg)\nonumber\\
  &\stackrel{\eqref{Dirichletcorrespondenceeq4}, \eqref{Jarcheq6}}{=}(-1)^kN^{\frac52-2s} f_{\scriptscriptstyle\hat N}^{3-2s}\eta( f_{\scriptscriptstyle\hat N}^2)\tilde H_{D,s,\eta}( e_{\scriptscriptstyle\hat N}, f_{\scriptscriptstyle\hat N})\frac{L(s-1,\chi_D\eta)}{L(s,\eta)L(2s-2,\eta^2)}\nonumber\\
    &\qquad\Bigg(\prod_{p<\infty}\varepsilon(1/2,\chi_p,\psi_p)\Bigg)\Bigg(\prod_{\substack{p<\infty\\p\mid N\\p\nmid r}}\chi_p(r)\Bigg)\Bigg(\prod_{\substack{p<\infty\\p\mid N\\p\mid r}}p^{n_p(2-s)}\chi_p(p^{n_p})K(s,T,\chi_p)\Bigg).
\end{align}
Note here that by $\chi_D\eta$ we mean the Dirichlet character mod $|D|N$ defined by $(\chi_D\eta)(x)=\chi_D(x)\eta(x)$ for $x$ relatively prime to both $|D|$ and~$N$. There is no guarantee that $\chi_D\eta$ is primitive. Similarly, $\eta^2$, which is a Dirichlet character $\bmod~N$, need not be primitive.

An exercise using the formula \eqref{epsilonfactorformula} shows that
\begin{equation}\label{rank2cT4eq43}
 \prod_{p<\infty}\varepsilon(1/2,\chi_p,\psi_p)=\frac{(-1)^kG(\eta)}{\sqrt{N}},
\end{equation} where $G(\eta)=\sum_{k=0}^{N-1}\eta(k)\ee^{2\pi\ii k/N}$ is a classical Gauss sum. Substituting into \eqref{rank2cT4eq32}, we get
\begin{align}\label{rank2cT4eq44}
  \prod_{p<\infty}I_p&=N^{2-2s} f_{\scriptscriptstyle\hat N}^{3-2s}\eta( f_{\scriptscriptstyle\hat N}^2)\tilde H_{D,s,\eta}( e_{\scriptscriptstyle\hat N}, f_{\scriptscriptstyle\hat N})G(\eta)\frac{L(s-1,\chi_D\eta)}{L(s,\eta)L(2s-2,\eta^2)}\nonumber\\
    &\qquad\Bigg(\prod_{\substack{p<\infty\\p\mid N\\p\nmid r}}\chi_p(r)\Bigg)\Bigg(\prod_{\substack{p<\infty\\p\mid N\\p\mid r}}p^{n_p(2-s)}\chi_p(p^{n_p})K(s,T,\chi_p)\Bigg).
\end{align}
Now we factor in the archimedean integral, which we recall from \eqref{rank2cT4realpropcoreq1} was only calculated under the condition $s=k$. Assuming $T$ is positive definite, we get from \eqref{rank2cT4eq1}, \eqref{rank2cT4realpropcoreq1} and \eqref{rank2cT4eq44}
\begin{align}\label{rank2cT4eq46}
  &c_{T,4}(b_Z,k,f^{(k)})=\det(Y)^{\frac k2}\frac{(4\pi)^{2k-1}\det(T)^{k-\frac32}\ee^{2\pi\ii\,{\rm Tr}(TZ)}}{2(2k-2)!}N^{2-2k} f_{\scriptscriptstyle\hat N}^{3-2k}\eta( f_{\scriptscriptstyle\hat N}^2)\tilde H_{D,k,\eta}( e_{\scriptscriptstyle\hat N}, f_{\scriptscriptstyle\hat N})\nonumber\\
    &\qquad\frac{L(k-1,\chi_D\eta)}{L(k,\eta)L(2k-2,\eta^2)}G(\eta)\Bigg(\prod_{\substack{p<\infty\\p\mid N\\p\nmid r}}\chi_p(r)\Bigg)\Bigg(\prod_{\substack{p<\infty\\p\mid N\\p\mid r}}p^{n_p(2-k)}\chi_p(p^{n_p})K(k,T,\chi_p)\Bigg).
\end{align}

\begin{theorem}\label{Fourierexpansiontheorem}
 Let $k\geq4$ be an integer, and let $\eta$ be a primitive Dirichlet character of conductor~$N=\prod p^{n_p}$. Then the function $E_{k,\eta}(Z)$ defined in \eqref{rewritesummationeq12} has a Fourier expansion
 \begin{equation}\label{Fourierexpansiontheoremeq1}
  E_{k,\eta}(Z)=\delta_{\eta=1}+\sum_{T\neq0}a(T)\ee^{2\pi\ii\trace(TZ)},
 \end{equation}
 where $T$ runs over non-zero positive semi-definite matrices $\mat{n}{r/2}{r/2}{m}$ with integers $n,r,m$ such that $N^2\mid m$, and where, for $\mathrm{rank}(T)=1$,
 \begin{align}\label{Fourierexpansiontheoremeq1a}
  a(T)&=\frac{(-2\pi \ii)^k}{(k-1)!}
  \begin{cases}
    \displaystyle\frac{ \sigma_{k-1}(n)}{\zeta(k)}  & \text{if }N=1,\:T=\mat{n}{0}{0}{0},\:n>0,\\[1ex] 
   \displaystyle \frac{\sigma_{k-1,\eta}(e_{\scriptscriptstyle \hat N})}{L(k,\eta)}\frac{\eta(r_{\scriptscriptstyle{\hat N}})}{\eta(2_{\scriptscriptstyle{\hat N}})}e_{\scriptscriptstyle N}^{k-1}
   &\text{if }m>0,\:r_{\scriptscriptstyle N}=\frac{(2m)_{\scriptscriptstyle N}}{N},\\
   0&\text{otherwise},
  \end{cases}
\end{align}
and for $\mathrm{rank}(T)=2$,
 \begin{align}\label{Fourierexpansiontheoremeq2}
  &a(T)=\frac{(4\pi)^{2k-1}\det(T)^{k-\frac32}}{2(2k-2)!}N^{2-2k} f_{\scriptscriptstyle\hat N}^{3-2k}\eta( f_{\scriptscriptstyle\hat N}^2)\tilde H_{D,k,\eta}( e_{\scriptscriptstyle\hat N}, f_{\scriptscriptstyle\hat N})\nonumber\\
    &\qquad\frac{L(k-1,\chi_D\eta)}{L(k,\eta)L(2k-2,\eta^2)}G(\eta)\Bigg(\prod_{\substack{p<\infty\\p\mid N\\p\nmid r}}\chi_p(r)\Bigg)\Bigg(\prod_{\substack{p<\infty\\p\mid N\\p\mid r}}p^{n_p(2-k)}\chi_p(p^{n_p})K(k,T,\chi_p)\Bigg).
 \end{align}
 Here, $r^2-4nm=Df^2$ as in Lemma~\ref{funddisclemma}, $e=\gcd(n,r,m)$, and we employ the notations \eqref{Nnotation} and \eqref{divisorsumwithcharacter}.
\end{theorem}
\begin{proof}
This follows from \eqref{rewritesummationeq12}, \eqref{rewritesummationeq10}, \eqref{Fourierexpansioneq2}, \eqref{FCcalceq1}, and the formulas referenced in table~\eqref{cTitable}.
\end{proof}

The local quantity $K(k,T,\chi_p)$ is defined in \eqref{rank2badeq4}, and summarized for odd $p$ and quadratic $\chi_p$ in Sect.~\ref{summaryofK}. The quantity $\tilde H_{D,k,\eta}( e_{\scriptscriptstyle\hat N}, f_{\scriptscriptstyle\hat N})$, which is $1$ for $f=1$, is defined in \eqref{tildeHetaeq1}.  

For $\eta=1$, using the functional equations for $\zeta(s)$ and $L(s,\chi_D)$, the formulas \eqref{Fourierexpansiontheoremeq1a} and \eqref{Fourierexpansiontheoremeq2} reduce to the one given in Corollary 2 to Theorem 6.3 of \cite{EichlerZagier1985}.
\section{The unramified calculation}\label{unramcalcsec}
In this section we fix a prime $p$. Let $|\cdot|$ be the $p$-adic absolute value, and let $v$ be the $p$-adic valuation. (This section is purely local, so we omit most subindices $p$.) Let $\psi$ be a character of~$\Q_p$ of conductor $\Z_p$.

Let $n,r,m\in\Z$ be such that $\Delta=4mn-r^2$ is non-zero. For simplicity we write $L$ for the quantity $\chi_D(p)$ defined in \eqref{Lpdefeq}, and write $d(\Delta)$ for the discriminant valuation $d_p(\Delta)$ defined in~\eqref{DeltaDeq2}. We will repeatedly make use of the identity
\begin{equation}\label{usefulidentity}
    \mat{1}{}{x}{1}=\mat{1}{x^{-1}}{}{1}\mat{-x^{-1}}{}{}{-x}\mat{}{1}{-1}{}\mat{1}{x^{-1}}{}{1}
\end{equation} as well as
\begin{equation}\label{psiintegralZp}
    \int\limits_{\Z_p}\psi(xy)dx=\begin{cases}
        1 & \text{if }v(y)\geq 0,\\
        0 & \text{if }v(y)< 0,
    \end{cases}
    \qquad\quad
    \int\limits_{\Z_p^\times}\psi(xy)dx=\begin{cases}
        1-p^{-1} & \text{if }v(y)\geq 0,\\
        -p^{-1} & \text{if }v(y)= -1,\\
        0 & \text{if }v(y)< -1.
    \end{cases}
\end{equation}
\subsection{Statement of result and initial reduction}
Let $f_s$ be the spherical vector in the local representation $|\cdot|^{s-3/2}1_{\GL(2,\Q_p)}\rtimes|\cdot|^{-s+3/2}$, normalized so that $f_s(1)=1$. In this section we will calculate the integral
\begin{align}\label{rank2goodeq1}
 I&:=\int\limits_{\Q_p}\int\limits_{\Q_p}\int\limits_{\Q_p}f_s(s_2s_1s_2\begin{bsmallmatrix}1&&\mu&\kappa\\&1&\lambda&\mu\\&&1\\&&&1\end{bsmallmatrix})\psi(n\lambda+r\mu+m\kappa)^{-1}\,d\lambda\,d\mu\,d\kappa\nonumber\\
 &=\int\limits_{\Q_p}\int\limits_{\Q_p}\int\limits_{\Q_p}f_s(\begin{bsmallmatrix}1\\&1\\\mu&\kappa&1\\\lambda&\mu&&1\end{bsmallmatrix})\psi(n\lambda+r\mu+m\kappa)\,d\lambda\,d\mu\,d\kappa,
\end{align}
appearing in~\eqref{rank2cT4eq1}. We write $I(n,r,m)$ or $I(T)$ if it is important to stress the dependence on $T=\mat{n}{r/2}{r/2}{m}$. A straightforward argument verifies that
\begin{equation}\label{rank2goodeq2}
 I(T)=I(AT\,^t\!A)\qquad\text{for all }A\in\GL(2,\Z_p).
\end{equation}
Our main result in this section is
\begin{theorem}\label{rank2goodtheorem}
 Let $n,r,m\in\Z$ be such that $\Delta=4mn-r^2\neq0$. Let $e=\min(v(n),v(r),v(m))$. Then
 \begin{align}\label{rank2goodtheoremeq1}
  I&=\frac{(1-p^{-s})(1-p^{2-2s})}{1-Lp^{1-s}}p^{\frac{v(\Delta)-d(\Delta)}2(3-2s)}\nonumber\\
  &\qquad\sum_{i=0}^{e}p^{i(s-1)}\Bigg(\sum_{j=0}^{\frac{v(\Delta)-d(\Delta)}2-i}p^{j(2s-3)}-Lp^{s-2}\sum_{j=0}^{\frac{v(\Delta)-d(\Delta)}2-i-1}p^{j(2s-3)}\Bigg).
 \end{align}
\end{theorem}
\begin{proof}
Note that both sides of the asserted equality are invariant under the transformation $T\mapsto AT\,^t\!A$ with $A\in\GL(2,\Z_p)$. We may therefore assume that $T$ is of a convenient form that can be achieved with this kind of transformation:
\begin{itemize}
 \item If $p$ is odd, there exists an $A$ such that $AT\,^t\!A$ is diagonal; see \S92 of \cite{OMeara1971}. Hence, to prove \eqref{rank2goodtheoremeq1}, we may assume that $r=0$.
 \item Assume that $p=2$. Then, by Lemma 4.1 of \cite{Cassels1978}, $T$ can either be diagonalized, or be brought into one of the forms $T=2^{e-1}\mat{2}{1}{1}{2}$ or $T=2^{e-1}\mat{}{1}{1}{}$ for some $e\geq1$.
\end{itemize}
In the interest of space, we will assume in all following calculations that $p$ is odd and hence that $r=0$. (See \cite{Pierce2025} for $p=2$.) We may then further assume that $v(m)\leq v(n)$, so that $e=v(m)$. The calculation of the integral $I$ is lengthy, so we proceed by splitting $I$ into various pieces.  These pieces will be calculated in the following sections. Summing up the terms in \eqref{rank2goodeq19d}, \eqref{rank2goodeq38}, \eqref{rank2goodeq43}, we get
\begin{align}\label{rank2goodeq50}
 I&=(1-p^{-s})\Bigg[\frac{(1-p^{(2-s)(e+1)})(1-p^{2-2s})}{(1-p^{2-s})(1-p^{3-2s})}\nonumber\\
 &\qquad-(1-p^{(1-s)(e+1)})p^{(1-s)(1-e)}p^{-(3-2s)v(2)}\frac{p^{(3-2s)\lfloor\frac{v(\Delta)+1}2\rfloor}+p^{2-s}p^{(3-2s)\lfloor\frac{v(\Delta)}2\rfloor}}{1-p^{3-2s}}\nonumber\\
   &\qquad+\frac{(1-p^{(1-s)(e+1)})}{1-p^{1-s}}|L|(L+1)p^{(1-s)(1-e)}p^{(3-2s)\frac{v(\Delta)}2}\Bigg].
\end{align}
Using the formula for the geometric series, one can verify that this matches \eqref{rank2goodtheoremeq1}.
\end{proof}

\subsection{The sets \texorpdfstring{$R(i,j)$}{} and \texorpdfstring{$S(i,j)$}{}}
The calculations in the following sections will naturally lead us to the volumes of certain sets of $p$-adic numbers, which we calculate in this section.  We define, for integers $i,j$,
\begin{align}
 \label{Rij0eq1}R(i,j)&=\left\{\mu\in\Z_p^\times\mid v(n+r\mu p^{-j}+m\mu^2p^{-2j})-i\geq0\right\},\\
 \label{Rij0eq2}S(i,j)&=\left\{\mu\in\Z_p^\times\mid v(n+r\mu p^{-j}+m\mu^2p^{-2j})-i=-1\right\}.
\end{align}
Note that
\begin{equation}
 \label{FC7eq53}R(i,j)=\bigsqcup_{k=i+1}^\infty S(k,j),\qquad 
 S(i,j)=R(i-1,j)\setminus R(i,j),
\end{equation}
so that calculating ${\rm vol}(R(i,j))$ is equivalent to calculating ${\rm vol}(S(i,j))$.

\begin{lemma}\label{Rij0lemma}
 Assume that $p$ is odd, $r=0$ and $n,m\neq0$. Then the volume of $R(i,j)$ is given as follows.
\begin{equation}\label{Rij0table}
 \begin{array}{cccc}
  \toprule
   \text{conditions}&&&{\rm vol}(R(i,j))\\
  \toprule
   2j<v(m)-v(n)&&i\leq v(n)&1-p^{-1}\\
   &&i>v(n)&0\\
   2j=v(m)-v(n)&-nm\in\Q_p^{\times2}&i\leq v(n)&1-p^{-1}\\
   &&i>v(n)&2p^{v(n)-i}\\
   &-nm\notin\Q_p^{\times2}&i\leq v(n)&1-p^{-1}\\
   &&i>v(n)&0\\
   2j>v(m)-v(n)&&i+2j\leq v(m)&1-p^{-1}\\
   &&i+2j>v(m)&0\\
  \bottomrule
 \end{array}
\end{equation}
\end{lemma}
\begin{proof}
We have
\begin{equation}\label{Rij0lemmaeq4}
 R(i,j)=\left\{\mu\in\Z_p^\times\mid v(n+m\mu^2p^{-2j})\geq i\right\}=\left\{\mu\in\Z_p^\times\;\Big\vert\; v\left(\frac{np^{2j}}m+\mu^2\right)\geq i+2j-v(m)\right\}.
\end{equation}
The cases where $2j\neq v(m)-v(n)$ are easy to see. We assume for the rest of the proof that $2j=v(m)-v(n)$.
In this case $\frac{np^{2j}}m$ is a unit.  If $-\frac{np^{2j}}m\notin\Z_p^{\times2}$ (equivalently, $-nm\notin\Q_p^{\times2}$), then $\frac{np^{2j}}m+\mu^2$ is always a unit, so that
\begin{equation}\label{Rij0lemmaeq6}
 {\rm vol}(R(i,j))={\rm vol}(\left\{\mu\in\Z_p^\times\mid 0\geq i+2j-v(m)\right\})
 =\begin{cases}
   1-p^{-1}&\text{if }v(m)\geq i+2j,\\
   0&\text{if }v(m)<i+2j.
  \end{cases}
\end{equation}
Assume that $-\frac{np^{2j}}m=u^2$ with $u\in\Z_p^\times$ (equivalently, $-nm\in\Q_p^{\times2}$). Then
\begin{align}\label{Rij0lemmaeq7}
  {\rm vol}(R(i,j))&={\rm vol}\left(\left\{\mu\in\Z_p^\times\mid v\left(\mu^2-u^2\right)\geq i+2j-v(m)\right\}\right)\nonumber\\
  &={\rm vol}\left(\left\{\mu\in\Z_p^\times\mid v\left(\mu^2-1\right)\geq i+2j-v(m)\right\}\right)\nonumber\\
  &={\rm vol}\left(\left\{\mu\in\Z_p^\times\setminus((1+p\Z_p)\cup(-1+p\Z_p))\mid 0\geq i+2j-v(m)\right\}\right)\nonumber\\
  &\qquad+2\,{\rm vol}\left(\left\{\mu\in1+p\Z_p\mid v\left(\mu^2-1\right)\geq i+2j-v(m)\right\}\right)\nonumber\\
  &=2\,{\rm vol}\left(\left\{\mu\in p\Z_p\mid v\left(\mu^2+2\mu\right)\geq i+2j-v(m)\right\}\right)+\begin{cases}
           1-3p^{-1}&\text{if }v(m)\geq i+2j,\\
           0&\text{if }v(m)<i+2j,\\
          \end{cases}\nonumber\\
  &=2\,{\rm vol}\left(\left\{\mu\in p\Z_p\mid v(\mu)\geq i+2j-v(m)\right\}\right)+\begin{cases}
           1-3p^{-1}&\text{if }v(m)\geq i+2j,\\
           0&\text{if }v(m)<i+2j,\\
          \end{cases}\nonumber\\
  &=2\begin{cases}
      p^{-1}&\text{if }i+2j-v(m)<1\\
      p^{v(m)-i-2j}&\text{if }i+2j-v(m)\geq1\\
     \end{cases}+\begin{cases}
           1-3p^{-1}&\text{if }v(m)\geq i+2j,\\
           0&\text{if }v(m)<i+2j,\\
          \end{cases}\nonumber\\
  &=\begin{cases}
           1-p^{-1}&\text{if }v(m)\geq i+2j,\\
           2p^{v(m)-i-2j}&\text{if }v(m)<i+2j.
          \end{cases}
\end{align} This concludes the proof.
\end{proof}  

The values in Table \eqref{Rij0table} can be encoded in the following power series.  The proofs are straightforward, so we omit the details.

\begin{proposition}[First generating series]\label{Rijprop}
Assume that $r=0$ and $n,m\neq0$. Then we have the formal power series
   \begin{equation}\label{Rijpropeq1}
    \sum_{j=0}^\infty {\rm vol}(R(i,-j))Y^j=A_1(i)+A_2(i)
   \end{equation}
   with
   \begin{align}
    \label{RijpropeqA1}A_1(i)&=\delta_{i\leq v(n)}(1-p^{-1})\frac{Y^{\max(0,\lfloor\frac{i-v(m)+1}2\rfloor)}}{1-Y}\\
    \label{RijpropeqA2}A_2(i)&=\delta_{v(m)\leq v(n)}\delta_{i>v(4n)}|L|(L+1)p^{v(2n)-i}Y^{\frac{v(n)-v(m)}2}
   \end{align}
   and
   \begin{equation}\label{Rijpropeq2}
    \sum_{i=1}^\infty\sum_{j=0}^\infty {\rm vol}(R(i,-j))X^iY^j=B_1+B_2
   \end{equation}
   with
   \begin{align}
    \label{RijpropeqB1}B_1&=(1-p^{-1})X\frac{1-X^{\min(v(n),v(m))}}{(1-X)(1-Y)}\\
     &+\delta_{v(m)\leq v(n)}(1-p^{-1})\frac{YX^{v(m)+1}\left(1-(YX^2)^{\lfloor\frac{v(n)-v(m)+1}2\rfloor}+X\left(1-(YX^2)^{\lfloor\frac{v(n)-v(m)}2\rfloor}\right)\right)}{(1-Y)(1-YX^2)},\nonumber\\
      \label{RijpropeqB2}B_2&=\delta_{v(m)\leq v(n)}|L|(L+1)p^{-v(2)-1}\frac{X^{v(4n)+1}}{1-p^{-1}X}Y^{\frac{v(n)-v(m)}2}.
   \end{align}
\end{proposition}

\begin{proposition}[Second generating series]\label{Rijprop2}
 Assume that $r=0$ and $n,m\neq0$. Then, with $e:=\min(v(n),v(m))$, we have the formal power series
 \begin{equation}\label{Rijprop2eq1}
  \sum_{j=1}^{\lfloor\frac{v(m)}2\rfloor}{\rm vol}(R(i,j))Y^j=C_1(i)+C_2(i)\qquad (i\geq0)
 \end{equation}
 with
 \begin{align}
  \label{Rijprop2eqC1}C_1(i)&=\delta_{i\leq e}(1-p^{-1})\frac{Y-Y^{\lfloor\frac{v(m)-i+2}2\rfloor}}{1-Y},\\
  \label{Rijprop2eqC2}C_2(i)&=\delta_{v(m)>v(n)}\delta_{i>v(4n)}|L|(L+1)p^{v(2n)-i}Y^{\frac{v(m)-v(n)}2}
   \end{align}
  and
  \begin{equation}\label{Rijprop2eq2}
   \sum_{i=0}^\infty\sum_{j=1}^{\lfloor\frac{v(m)}2\rfloor}{\rm vol}(R(i,j))X^iY^j=D_1+D_2
  \end{equation}
  with
  \begin{align}
   \label{Rijprop2eqD1}D_1&=(1-p^{-1})Y\frac{1-X^{e+1}}{(1-X)(1-Y)}\\
      &\qquad-(1-p^{-1})\frac{Y^{\lfloor\frac{v(m)+1}2\rfloor}X(1-(X^2Y^{-1})^{\lfloor\frac{e+1}2\rfloor})+Y^{\lfloor\frac{v(m)+2}2\rfloor}(1-(X^2Y^{-1})^{\lfloor\frac{e+2}2\rfloor})}{(1-Y)(1-X^2Y^{-1})}\nonumber\\
   \label{Rijprop2eqD2}D_2&=\delta_{v(m)>v(n)}|L|(L+1)p^{-v(2)-1}\frac{X^{v(4n)+1}}{1-p^{-1}X}Y^{\frac{v(m)-v(n)}2}.
  \end{align}
\end{proposition}

\subsection{Calculation of \texorpdfstring{$I_1$}{}}
Recall that our goal is to calculate \eqref{rank2goodeq1}.
 As mentioned above, we assume that $p$ is odd, $r=0$, and $n,m\neq0$. Let $I_1$ be the part of the integral \eqref{rank2goodeq1} where $\lambda\in\Z_p$, and let $I_2$ be the part where $\lambda\notin\Z_p$. Clearly,
\begin{align}\label{rank2goodeq11}
 I_1&=\int\limits_{\Q_p}\int\limits_{\Q_p}f_s(\begin{bsmallmatrix}1\\&1\\\mu&\kappa&1\\&\mu&&1\end{bsmallmatrix})\psi(m\kappa)\,d\mu\,d\kappa.
\end{align}
Let $I_{11}$ be the part where $\mu\in\Z_p$, and let $I_{12}$ be the part where $\mu\notin\Z_p$. A straightforward calculation using \eqref{usefulidentity} gives
\begin{align}\label{rank2goodeq12}
 I_{11}=\frac{(1-p^{-s})(1-p^{(1-s)(v(m)+1)})}{1-p^{1-s}}.
\end{align}and
\begin{align}\label{rank2goodeq13}
    I_{12}&=\int\limits_{\Q_p}\int\limits_{\Q_p\setminus\Z_p}|\mu|^{2-2s}f_s(\begin{bsmallmatrix}1\\&1\\&\kappa&1\\&&&1\end{bsmallmatrix})\psi(m\kappa\mu^2)\,d\mu\,d\kappa.
\end{align}
Let $I_{121}$ be the part of \eqref{rank2goodeq13} where $\kappa\in\Z_p$, and let $I_{122}$ be the part where $\kappa\notin\Z_p$. We have
\begin{align}\label{rank2goodeq14}
 I_{121}&=\sum_{j=1}^\infty\:\int\limits_{\Z_p}\int\limits_{p^{-j}\Z_p^\times}|\mu|^{3-2s}\psi(m\kappa\mu^2)\,d^\times\mu\,d\kappa=\sum_{j=1}^\infty\:\int\limits_{\Z_p}\int\limits_{\Z_p^\times}p^{j(3-2s)}\psi(m\kappa\mu^2p^{-2j})\,d^\times\mu\,d\kappa\nonumber\\
 &=\sum_{j=1}^{\lfloor\frac{v(m)}2\rfloor}\int\limits_{\Z_p^\times}p^{j(3-2s)}\,d^\times\mu=(1-p^{-1})p^{3-2s}\frac{1-p^{\lfloor\frac{v(m)}2\rfloor(3-2s)}}{1-p^{3-2s}}.
\end{align}
Next,
\begin{align}\label{rank2goodeq18}
 I_{122}&=\iint\limits_{(\Q_p\setminus\Z_p)^2}|\mu|^{2-2s}f_s(\begin{bsmallmatrix}1\\&1&\kappa^{-1}\\&&1\\&&&1\end{bsmallmatrix}\begin{bsmallmatrix}1\\&-\kappa^{-1}\\&&-\kappa\\&&&1\end{bsmallmatrix})\psi(m\kappa\mu^2)\,d\mu\,d\kappa\nonumber\\
 &=\iint\limits_{(\Q_p\setminus\Z_p)^2}|\mu|^{2-2s}|\kappa|^{-s}\psi(m\kappa\mu^2)\,d\mu\,d\kappa=\sum_{i=1}^\infty\sum_{j=1}^\infty\:\iint\limits_{(\Z_p^\times)^2}p^{j(3-2s)}p^{i(1-s)}\psi(m\kappa\mu^2p^{-i-2j})\,d^\times\mu\,d^\times\kappa\nonumber\\
    &=(1-p^{-1})^2\frac{p^{4-3s}}{1-p^{1-s}}\,\frac{1-p^{(3-2s)\lfloor\frac{v(m)-1}2\rfloor}}{1-p^{3-2s}}+(1-p^{-1})\,\frac{p^{(v(m)+1)(1-s)}}{1-p^{1-s}}\,(1-p^{\lfloor\frac{v(m)-1}2\rfloor})\nonumber\\
   &\qquad+p^{v(m)(1-s)-s}(1-p^{\lfloor\frac{v(m)}2\rfloor}).
\end{align}
Adding up the terms gives
\begin{align}\label{rank2goodeq19d}
 I_1&=\frac{1-p^{-s}}{1-p^{1-s}}+(1-p^{-1})p^{3-2s}\frac{1-p^{\lfloor\frac{e}2\rfloor(3-2s)}}{1-p^{3-2s}}+(1-p^{-1})^2\frac{p^{4-3s}}{1-p^{1-s}}\,\frac{1-p^{(3-2s)\lfloor\frac{e-1}2\rfloor}}{1-p^{3-2s}}\nonumber\\
   &\qquad-(1-p^{-1})\,\frac{p^{(e+1)(1-s)+\lfloor\frac{e-1}2\rfloor}}{1-p^{1-s}}-p^{e(1-s)-s+\lfloor\frac{e}2\rfloor}.
\end{align}
\subsection{Calculation of \texorpdfstring{$I_{21}$}{}}
By applying \eqref{usefulidentity} to $\lambda$ and commuting matrices appropriately, we get
\begin{align}\label{rank2goodeq20}
 I_2&=\int\limits_{\Q_p}\int\limits_{\Q_p}\int\limits_{\Q_p\setminus\Z_p}|\lambda|^{1-s}f_s(\begin{bsmallmatrix}1\\&1\\\mu&\kappa&1\\&\mu&&1\end{bsmallmatrix})\psi((n+m\mu^2)\lambda+m\kappa)\,d\lambda\,d\mu\,d\kappa.
\end{align}
Let $I_{21}$ be the part where $\mu\in\Z_p$, and let $I_{22}$ be the part where $\mu\notin\Z_p$.
We have
\begin{align}\label{rank2goodeq30}
 I_{21}&=\int\limits_{\Q_p}\int\limits_{\Z_p}\int\limits_{\Q_p\setminus\Z_p}|\lambda|^{1-s}f_s(\begin{bsmallmatrix}1\\&1\\&\kappa&1\\&&&1\end{bsmallmatrix})\psi((n+m\mu^2)\lambda+m\kappa)\,d\lambda\,d\mu\,d\kappa.
\end{align}
Let $I_{211}$ be the part of this integral where $\kappa\in\Z_p$, and let $I_{212}$ be the part where $\kappa\notin\Z_p$. Clearly,
\begin{align}\label{rank2goodeq31}
 I_{211}&=\int\limits_{\Z_p}\int\limits_{\Q_p\setminus\Z_p}|\lambda|^{1-s}\psi((n+m\mu^2)\lambda)\,d\lambda\,d\mu.
\end{align}
Using \eqref{usefulidentity} on $\kappa$ and \eqref{psiintegralZp},
\begin{align}\label{rank2goodeq32}
 I_{212}&=\frac{(1-p^{-1})p^{1-s}-(1-p^{-s})p^{(v(m)+1)(1-s)}}{1-p^{1-s}}I_{211}.
\end{align}
Furthermore,
\begin{align}\label{rank2goodeq34}
 I_{211}&=\int\limits_{\Z_p}\int\limits_{\Q_p\setminus\Z_p}|\lambda|^{1-s}\psi((n+m\mu^2)\lambda)\,d\lambda\,d\mu\nonumber\\
&=\sum_{j=0}^\infty\sum_{i=1}^\infty\:\int\limits_{\Z_p^\times}\int\limits_{\Z_p^\times}p^{i(2-s)}p^{-j}\psi((n+m\mu^2p^{2j})p^{-i}\lambda)\,d^\times\lambda\,d^\times\mu\nonumber\\
 &\overset{\eqref{psiintegralZp}}{=}(1-p^{-1})\sum_{j=0}^\infty\sum_{i=1}^\infty\:\int\limits_{R(i,-j)}p^{i(2-s)}p^{-j}\,d^\times\mu-p^{-1}\sum_{j=0}^\infty\sum_{i=1}^\infty\:\int\limits_{S(i,-j)}p^{i(2-s)}p^{-j}\,d^\times\mu\nonumber\\
 &\stackrel{\eqref{FC7eq53}}{=}(1-p^{1-s})\sum_{j=0}^\infty\sum_{i=1}^\infty p^{i(2-s)}p^{-j}\,{\rm vol}(R(i,-j))-p^{1-s}\sum_{j=0}^\infty p^{-j}\,{\rm vol}(R(0,-j)).
\end{align}
By Proposition~\ref{Rijprop}, after a calculation and assuming $v(m)\leq v(n)$, we get
\begin{align}\label{rank2goodeq37}
 I_{211}&=-p^{1-s}\frac{1-p+(1-p^{1-s})p^{(2-s)e+1}}{1-p^{2-s}}\nonumber\\
   &\quad+(1-p^{1-s})\frac{p^{(2-s)e+1-s}\left(1-p^{(3-2s)(\lfloor\frac{v(\Delta)+1}2\rfloor-e-v(2))}+p^{2-s}\left(1-p^{(3-2s)(\lfloor\frac{v(\Delta)}2\rfloor-e-v(2))}\right)\right)}{1-p^{3-2s}}\nonumber\\
   &\quad+|L|(L+1)p^{(1-s)(1-e)}p^{(3-2s)\frac{v(\Delta)}2}.
\end{align}
To summarize, if $v(m)\leq v(n)$ then
\begin{align}\label{rank2goodeq38}
 I_{21}&=\frac{(1-p^{-s})(1-p^{(v(m)+1)(1-s)})}{1-p^{1-s}}I_{211}
\end{align}
with $I_{211}$ given by \eqref{rank2goodeq37}.
\subsection{Calculation of \texorpdfstring{$I_{22}$}{}}
By applying \eqref{usefulidentity} to $\mu$ and commuting matrices appropriately,
\begin{align}\label{rank2goodeq21}
I_{22}&=\int\limits_{\Q_p}\int\limits_{\Q_p\setminus\Z_p}\int\limits_{\Q_p\setminus\Z_p}|\lambda|^{1-s}|\mu|^{2-2s}f_s(\begin{bsmallmatrix}1\\&1\\&\kappa&1\\&&&1\end{bsmallmatrix})\psi((n+m\mu^2)\lambda+m\kappa\mu^2)\,d\lambda\,d\mu\,d\kappa.
\end{align}
Let $I_{221}$ be the part of this integral where $\kappa\in\Z_p$, and let $I_{222}$ be the part where $\kappa\notin\Z_p$. Then
\begin{align}\label{rank2goodeq40}
 I_{221}&=\sum_{j=1}^{\lfloor\frac{v(m)}2\rfloor}\:\int\limits_{\Z_p^\times}\int\limits_{\Q_p\setminus\Z_p}|\lambda|^{1-s}p^{j(3-2s)}\psi((n+m\mu^2p^{-2j})\lambda)\,d\lambda\,d^\times\mu\nonumber\\
 &=\sum_{i=1}^\infty\sum_{j=1}^{\lfloor\frac{v(m)}2\rfloor}\:\int\limits_{\Z_p^\times}\int\limits_{\Z_p^\times}p^{i(2-s)}p^{j(3-2s)}\psi((n+m\mu^2p^{-2j})\lambda p^{-i})\,d^\times\lambda\,d^\times\mu\nonumber\\
 &\overset{\eqref{psiintegralZp}}{=}(1-p^{-1})\sum_{i=1}^\infty\sum_{j=1}^{\lfloor\frac{v(m)}2\rfloor}p^{i(2-s)}p^{j(3-2s)}{\rm vol}(R(i,j))-p^{-1}\sum_{i=1}^\infty\sum_{j=1}^{\lfloor\frac{v(m)}2\rfloor}p^{i(2-s)}p^{j(3-2s)}{\rm vol}(S(i,j))\nonumber\\
 &\stackrel{\eqref{FC7eq53}}{=}(1-p^{1-s})\sum_{i=1}^\infty\sum_{j=1}^{\lfloor\frac{v(m)}2\rfloor}p^{i(2-s)}p^{j(3-2s)}{\rm vol}(R(i,j))-p^{1-s}\sum_{j=1}^{\lfloor\frac{v(m)}2\rfloor}p^{j(3-2s)}{\rm vol}(R(0,j)).
\end{align}
With similar calculations we get
\begin{align}\label{rank2goodeq41}
 I_{222}&=(1-p^{-1})p^{1-s}\sum_{i=1}^\infty\sum_{j=1}^{\lfloor\frac{v(m)-1}2\rfloor}p^{i(2-s)}p^{j(3-2s)}{\rm vol}(R(i,j))\nonumber\\
   &\quad-(1-p^{-1})p^{(1-s)(v(m)+1)}\sum_{i=1}^\infty\sum_{j=1}^{\lfloor\frac{v(m)-1}2\rfloor}p^{i(2-s)}p^{j}{\rm vol}(R(i,j))\nonumber\\
   &\quad-(1-p^{-1})\frac{p^{2-2s}}{1-p^{1-s}}\sum_{j=1}^{\lfloor\frac{v(m)-1}2\rfloor}p^{j(3-2s)}{\rm vol}(R(0,j))\nonumber\\
   &\quad+(1-p^{-1})\frac{p^{(1-s)(v(m)+2)}}{1-p^{1-s}}\sum_{j=1}^{\lfloor\frac{v(m)-1}2\rfloor}p^{j}{\rm vol}(R(0,j))\nonumber\\
   &\quad-p^{(v(m)+1)(1-s)-1}(1-p^{1-s})\sum_{i=1}^\infty\sum_{j=1}^{\lfloor\frac{v(m)}2\rfloor}p^{i(2-s)}p^{j}\:{\rm vol}(R(i,j))\nonumber\\
   &\quad+p^{(v(m)+1)(1-s)-s}\sum_{j=1}^{\lfloor\frac{v(m)}2\rfloor}p^{j}\,{\rm vol}(R(0,j)).
\end{align}
Adding $I_{221}$ from \eqref{rank2goodeq40}, we get, after a calculation,
\begin{align}\label{rank2goodeq42}
 I_{22}&=\frac{1-p^{-s}}{1-p^{1-s}}\Bigg[(1-p^{1-s})\sum_{i=0}^\infty\sum_{j=1}^{\lfloor\frac{v(m)}2\rfloor}p^{i(2-s)}p^{j(3-2s)}{\rm vol}(R(i,j))\nonumber\\
   &\quad-(1-p^{1-s})p^{(1-s)(v(m)+1)}\sum_{i=0}^\infty\sum_{j=1}^{\lfloor\frac{v(m)}2\rfloor}p^{i(2-s)}p^{j}{\rm vol}(R(i,j))\nonumber\\
   &\quad-\sum_{j=1}^{\lfloor\frac{v(m)}2\rfloor}p^{j(3-2s)}{\rm vol}(R(0,j))+p^{(1-s)(v(m)+1)}\sum_{j=1}^{\lfloor\frac{v(m)}2\rfloor}p^{j}{\rm vol}(R(0,j))\Bigg].
\end{align}
Using \eqref{Rijprop2eq1} and \eqref{Rijprop2eq2}, and assuming $v(m)\leq v(n)$, we get
\begin{align}\label{rank2goodeq43}
 &I_{22}=\frac{1-p^{-s}}{1-p^{1-s}}p^{2-s}\Bigg[(1-p^{1-s})(1-p^{-1})\Bigg(p^{1-s}\frac{1-p^{(2-s)(e+1)}}{(1-p^{2-s})(1-p^{3-2s})}-p^{(1-s)e}\frac{1-p^{(2-s)(e+1)}}{(1-p^{2-s})(1-p)}\nonumber\\
      &\qquad+(1-p^{(1-s)(e+1)})\frac{p^{(4-2s)\lfloor\frac{e+1}2\rfloor}+p^{(4-2s)\lfloor\frac{e}2\rfloor+2-s}}{(1-p^{3-2s})(1-p)}\Bigg)\nonumber\\
   &\qquad-(1-p^{-1})p^{1-s}\frac{1-p^{(3-2s)\lfloor\frac{e}2\rfloor}}{1-p^{3-2s}}-p^{(1-s)e-1}(1-p^{\lfloor\frac{e}2\rfloor})\Bigg].
\end{align}
\section{The ramified calculation}\label{ramcalcsec}
As in the previous section we let $|\cdot|$ be the $p$-adic absolute value, and let $v$ be the $p$-adic valuation. Let $\psi$ be a character of $\Q_p$ of conductor $\Z_p$.
\subsection{Support lemmas}\label{supportlemmassec}
Let $n$ be a positive integer. By Proposition~5.1.2 of \cite{NF}, we have, using the notation~\eqref{C0defeq},
\begin{equation}\label{supportlemmaeq1}
 \GSp(4,\Q_p)=\bigsqcup_{i=0}^nP(\Q_p)\mathcal{C}_0(p^i)K(\p^{2n}).
\end{equation}
By the last comment in Sect.~5.1 of \cite{NF},
\begin{equation}\label{supportlemmaeq2}
 P(\Q_p)1K(\p^{2n})=P(\Q_p)\mathcal{C}_0(1)K(\p^{2n})=P(\Q_p)s_2s_1s_2K(\p^{2n}).
\end{equation}
The following lemma describes when certain elements are contained in the ``last'' double coset $P(\Q_p)\mathcal{C}_0(p^n)K(\p^{2n})$.
\begin{lemma}\label{supportlemma1}  Suppose $n>0$.
\begin{enumerate}
    \item We have \begin{equation}\label{supportlemma1eq1}
  \begin{bsmallmatrix}1\\&1\\&\kappa&1\\&&&1\end{bsmallmatrix}\notin P(\Q_p)\mathcal{C}_0(p^n)K(\p^{2n}) \quad \text{ and }\quad \begin{bsmallmatrix}1\\&1\\&&1\\\kappa&&&1\end{bsmallmatrix}\notin P(\Q_p)\mathcal{C}_0(p^n)K(\p^{2n})
 \end{equation}
 for any $\kappa\in\Q_p$.
 \item We have \begin{equation}\label{supportlemma4eq1}
  \begin{bsmallmatrix}1\\&1\\\mu&&1\\\lambda&\mu&&1\end{bsmallmatrix}\in P(\Q_p)\mathcal{C}_0(p^n)K(\p^{2n})\quad\Longleftrightarrow\quad v(\mu)=n\text{ and }v(\lambda)\geq2n.
 \end{equation}
 \item We have \begin{equation}\label{supportlemma6eq1}
  s_2s_1s_2\begin{bsmallmatrix}1&&\mu&\kappa\\&1&\lambda&\mu\\&&1\\&&&1\end{bsmallmatrix}\in P(\Q_p)\mathcal{C}_0(p^n)K(\p^{2n}) \end{equation}
 if and only if one of the following conditions is satisfied:
 \begin{enumerate}
  \item $v(\lambda)\geq0$ and $v(\mu)=-n$ and $v(\kappa)\geq-2n$. In this case we have the identity
  \begin{align}\label{supportlemma6eq2}
   &s_2s_1s_2\begin{bsmallmatrix}1&&\mu&\kappa\\&1&\lambda&\mu\\&&1\\&&&1\end{bsmallmatrix}=\begin{bsmallmatrix}&1\\p^{2n}\\&&&p^{-2n}\\&&1\end{bsmallmatrix}\begin{bsmallmatrix}1\\&1\\-\mu p^{2n}&&1\\&-\mu p^{2n}&&1\end{bsmallmatrix}\begin{bsmallmatrix}&&&p^{-2n}\\&&1\\&-1\\-p^{2n}\end{bsmallmatrix}\begin{bsmallmatrix}1&&&\kappa\\&1&\lambda\\&&1\\&&&1\end{bsmallmatrix}.
  \end{align}
  \item $v(\lambda)<0$ and $v(\mu)=-n+v(\lambda)$ and $v(\kappa-\mu^2\lambda^{-1})\geq-2n$. In this case we have the identity
  \begin{align}\label{supportlemma6eq3}
   &s_2s_1s_2\begin{bsmallmatrix}1&&\mu&\kappa\\&1&\lambda&\mu\\&&1\\&&&1\end{bsmallmatrix}=\begin{bsmallmatrix}-p^{2n}\mu\lambda^{-1}&\lambda^{-1}&-1\\p^{2n}\\&&\mu&p^{-2n}\\&&\lambda\end{bsmallmatrix}\nonumber\\
     &\qquad\begin{bsmallmatrix}1\\&1\\-p^{2n}\mu\lambda^{-1}&&1\\-p^{4n}(\kappa-\mu^2\lambda^{-1})&-p^{2n}\mu\lambda^{-1}&&1\end{bsmallmatrix}\begin{bsmallmatrix}&&&p^{-2n}\\&-1\\&-\lambda^{-1}&-1\\-p^{2n}\end{bsmallmatrix}.
  \end{align}
 \end{enumerate}
\end{enumerate}
\end{lemma}
\begin{proof}
    This is an exercise using \eqref{usefulidentity} and noticing $1,s_2s_1s_2\not\in P(\Q_p)\mathcal{C}_0(p^n)K(\p^{2n})$.
\end{proof}
\subsection{The main calculation}
In this section we assume that $p$ is a prime with $p\mid N$, equivalently, $n_p>0$. Let $\chi$ be a character of $\Q_p^\times$ with conductor exponent $a(\chi)=n_p$.  Let $f_s$ be the paramodular vector in the local representation $|\cdot|^{s-3/2}1_{\GL(2,\Q_p)}\rtimes|\cdot|^{-s+3/2}$, normalized as in~\eqref{PMieq2}. Our goal is to calculate
\begin{equation}\label{rank2badeq1}
 I:=\int\limits_{\Q_p}\int\limits_{\Q_p}\int\limits_{\Q_p}f_s(s_2s_1s_2\begin{bsmallmatrix}1&&\mu&\kappa\\&1&\lambda&\mu\\&&1\\&&&1\end{bsmallmatrix})\psi(n\lambda+r\mu+m\kappa)^{-1}\,d\lambda\,d\mu\,d\kappa.
\end{equation} appearing in~\eqref{rank2cT4eq1}. By Lemma~\ref{supportlemma1}~iii), $I=I_1+I_2$ with
\begin{align}
 \label{rank2badeq20a}I_1&=\int\limits_{p^{-2n_p}\Z_p}\,\int\limits_{p^{-n_p}\Z_p^\times}\, \int\limits_{\Z_p}f_s(\begin{bsmallmatrix}
        &1\\p^{2n_p}\\&&&p^{-2n_p}\\&&1
    \end{bsmallmatrix}\begin{bsmallmatrix}
        1\\&1\\-\mu p^{2n_p}&&1\\&-\mu p^{2n_p}&&1
    \end{bsmallmatrix})\psi(n\lambda+r\mu+m\kappa)^{-1} d\lambda\, d\mu\, d\kappa,\\
 \label{rank2badeq20b}I_2&=\int\limits_{\Q_p\setminus\Z_p}\,\int\limits_{\lambda p^{-n_p}\Z_p^\times}\,\int\limits_{\mu^2\lambda^{-1} +p^{-2n_p}\Z_p} f_s(\begin{bsmallmatrix}-p^{2n_p}\mu\lambda^{-1}&\lambda^{-1}&-1\\p^{2n_p}\\&&\mu&p^{-2n_p}\\&&\lambda\end{bsmallmatrix}\nonumber\\
      &\hspace{25ex}\begin{bsmallmatrix}
        1\\&1\\-\mu\lambda^{-1} p^{2n_p}&&1\\&-\mu\lambda^{-1} p^{2n_p}&&1
    \end{bsmallmatrix})\psi(n\lambda+r\mu+m\kappa)^{-1}d\kappa\, d\mu\, d\lambda.
\end{align}
From \eqref{indrepeq3} we calculate easily
\begin{equation}\label{rank2badeq21}
 I_1=\delta_{v(m)\geq 2n_p}\,p^{n_p(3-2s)}\chi(p^{n_p})\int\limits_{\Z_p^\times}\chi(\mu^{-1})\psi(r\mu p^{-n_p})^{-1} d\mu.
\end{equation}
This is zero if $r=0$ (because $\chi$ is ramified). It is known that (for any ramified character $\chi$ and additive character $\psi$ with conductor $\Z_p$)
\begin{equation}\label{epsilonfactorformula}
 \int\limits_{\Z_p^\times}\chi(\mu^{-1})\psi(x\mu)\,d\mu
 =\begin{cases}
   0&\text{if }a(\chi)\neq-v(x),\\
   p^{-a(\chi)/2}\chi(x)\varepsilon(1/2,\chi,\psi)&\text{if }a(\chi)=-v(x).
  \end{cases}
\end{equation}
Recall that $a(\chi)=n_p$. Hence, assuming $r\neq0$,
\begin{equation}\label{rank2badeq22}
 I_1=\delta_{v(m)\geq 2n_p}\,p^{n_p(5/2-2s)}\chi(-r)\varepsilon(1/2,\chi,\psi).
\end{equation}
Furthermore,
\begin{align}\label{rank2badeq2}
 I_2&=\delta_{v(m)\geq 2n_p}\, p^{n_p(3-2s)}\chi(p^{n_p})\int\limits_{\Q_p\setminus\Z_p}\,\int\limits_{\Z_p^\times}|\lambda|^{1-s}\chi(\lambda^{-1}\mu^{-1})\nonumber\\
    &\hspace{25ex} \psi(n\lambda +r\mu\lambda p^{-n_p} +m\mu^2\lambda p^{-2n_p})^{-1} d\mu\, d\lambda\nonumber\\
%    &=\delta_{v(m)\geq 2n_p}\, p^{n_p(3-2s)}\chi(p^{n_p})\sum_{i=1}^{\infty}\int\limits_{p^{-i}\Z_p^\times}\,\int\limits_{\Z_p^\times}|\lambda|^{2-s}\chi(\lambda^{-1})\chi(\mu^{-1})\nonumber\\
%    &\hspace{25ex} \psi((n +r\mu p^{-n_p} +m\mu^2 p^{-2n_p})\lambda)^{-1} d\mu\, d^\times\lambda\nonumber\\
    &=\delta_{v(m)\geq 2n_p}\, p^{n_p(3-2s)}\chi(p^{n_p})\sum_{i=1}^{\infty} p^{i(2-s)}\int\limits_{\Z_p^\times}\,\int\limits_{\Z_p^\times}\chi(p^i\lambda^{-1}\mu^{-1})\nonumber\\
    &\hspace{25ex} \psi((n +r\mu p^{-n_p} +m\mu^2 p^{-2n_p})\lambda p^{-i})^{-1} d\mu\, d^\times\lambda\nonumber\\
    &=\delta_{v(m)\geq 2n_p}\, p^{n_p(3-2s)}\chi(p^{n_p})\sum_{i=1}^{\infty} p^{i(2-s)}\chi(p^i)\sum_{j=-n_p}^{\infty}\:\int\limits_{\{\mu\in\Z_p^\times\mid v(n+r\mu p^{-n_p}+m\mu^2 p^{-2n_p})=j\}}\chi(\mu^{-1})\nonumber\\
    &\hspace{10ex}\int\limits_{\Z_p^\times}\chi(\lambda^{-1})\psi(-(n +r\mu p^{-n_p} +m\mu^2 p^{-2n_p})\lambda p^{-i}) d^\times\lambda\, d\mu\nonumber\\
    &\stackrel{\eqref{Rij0eq2},\,\eqref{epsilonfactorformula}}{=}\delta_{v(m)\geq 2n_p}\, p^{n_p(3-2s)}\chi(p^{n_p})\sum_{i=1}^{\infty} p^{i(2-s)}\chi(p^i)\sum_{j=1-n_p}^{\infty}\:\int\limits_{S(j+1,n_p)}\chi(\mu^{-1})\nonumber\\
    &\hspace{10ex}\begin{cases}
        0 & -(j-i)\neq n_p\\
        p^{-n_p/2}\chi(-(n+r\mu p^{-n_p}+m\mu^2 p^{-2n_p})p^{-i})\varepsilon(1/2,\chi,\psi) & -(j-i)=n_p
    \end{cases} d\mu\nonumber\\
    &=\delta_{v(m)\geq 2n_p}\, p^{n_p(3-2s)}\chi(p^{n_p})\sum_{j=1-n_p}^{\infty}\:\int\limits_{S(j+1,n_p)}\chi(\mu^{-1})\nonumber\\
    &\hspace{10ex}p^{(n_p+j)(2-s)} p^{-n_p/2}\chi(-(n+r\mu p^{-n_p}+m\mu^2 p^{-2n_p}))\varepsilon(1/2,\chi,\psi)d\mu\nonumber\\
%     &=\delta_{v(m)\geq 2n_p}\delta_{v(r)\neq0}\,p^{n_p(5/2-2s)}\chi(p^{n_p})\sum_{j=1-n_p}^{\infty} p^{(n_p+j)(2-s)}\nonumber\\
%     &\quad \int\limits_{S(j+1,n_p)}\chi(-(n\mu^{-1}+rp^{-n_p}+m\mu p^{-2n_p}))\varepsilon(1/2,\chi,\psi)d\mu\nonumber\\
    &=\delta_{v(m)\geq 2n_p}\,\delta_{v(r)\neq0}\,p^{n_p(9/2-3s)}\varepsilon(1/2,\chi,\psi)\chi(-p^{n_p})K(s,T,\chi)
\end{align}
with
\begin{equation}\label{rank2badeq4}
 K(s,T,\chi)=\sum_{j=1-n_p}^\infty\;p^{j(2-s)}\int\limits_{S(j+1,n_p)}\chi\left(n\mu^{-1}+rp^{-n_p}+m\mu p^{-2n_p}\right)\,d\mu.
\end{equation}
Hence, adding \eqref{rank2badeq22} and \eqref{rank2badeq2}, we get
\begin{align}\label{rank2badeq5}
    I&=\delta_{v(m)\geq 2n_p}\,p^{n_p(5/2-2s)}\varepsilon(1/2,\chi,\psi)\bigg(\delta_{v(r)=0}\, \chi(-r)+\delta_{v(r)\neq0}\, p^{n_p(2-s)}\chi(-p^{n_p})K(s,T,\chi)\bigg).
\end{align}
Hence the calculation of $I$ is reduced to that of $K(s,T,\chi)$.
\subsection{Summary of \texorpdfstring{$K(s,T,\chi)$}{} calculation}\label{summaryofK}
Typically, for a specific character $\chi$ and given $n,r,m$ in $\Z_p$, the quantity $K(s,T,\chi)$ defined in \eqref{rank2badeq4} is computable without too much effort. However, it seems difficult to obtain a closed formula valid for any $\chi$. Table~\eqref{KsTchitable} below gives the values of $K(s,T,\chi)$ in case $p$ is odd and $\chi$ is quadratic. We will carry out the proofs only in the case $r=0$, which demonstrates the methods. In this case, as we shall see in the next section, the value of $K(s,T,\chi)$ involves the elliptic curve point count
\begin{equation}\label{apdefeq}
 a_p:=p-\#\{(x,y)\in\F_p\times \F_p\mid y^2=x^3-Dx\},
\end{equation}
where we write $r^2-4mn=Df^2$ as in Lemma~\ref{funddisclemma} with a fundamental discriminant~$D$. For other cases and more proofs see~\cite{Pierce2025}. In the table, the quantity $e$ stands for $\min(v(n),v(r),v(m))$.
\begin{equation}\label{KsTchitable}
 \begin{array}{ccccccc}
  \toprule
   &&\text{conditions}&K(s,T,\chi)\\
  \toprule
   r=0&\chi(-1)=-1&&0\\
   &\chi(-1)=1&v(n)=v(m)-2n_p&-a_p p^{e(2-s)-1}\chi(2pf)\\
   &&v(n)\neq v(m)-2n_p&0\\
   r\neq0&&v(n)<\min(v(r)-n_p,v(m)-2n_p)&0\\
   &&v(r)-n_p<\min(v(n),v(m)-2n_p)&p^{e(2-s)}(1-p^{-1})\chi(pr)\\
   &&v(m)-2n_p<\min(v(n),v(r)-n_p)&0\\
   &&v(n)=v(r)-n_p<v(m)-2n_p&-p^{e(2-s)-1}\chi(pr)\\
   &&v(n)=v(m)-2n_p<v(r)-n_p&-a_p p^{e(2-s)-1}\chi(2pf)\\
   &&v(r)-n_p=v(m)-2n_p<v(n)&-p^{e(2-s)-1}\chi(pr)\\
   &&v(n)=v(r)-n_p=v(m)-2n_p&\eqref{Ksummaryeq2}\\
  \bottomrule
 \end{array}
\end{equation}
For the last line in the table (the case $r\neq 0$ and $v(n)=v(r)-1=v(m)-2$) we have
\begin{align}\label{Ksummaryeq2}
    &K(s,T,\chi)\nonumber\\
    &=\begin{cases}
  \displaystyle\frac{\chi(-2rp)(1-\chi(-D)p^{s-2})}{1-p^{3-2s}}(p^{e(2-s)+3-2s}(1+\displaystyle\chi(-D)p^{s-2})&\\[2ex]
    \displaystyle\qquad-p^{e(s-1)+v(f)(3-2s)}(1-\chi(-D)p^{1-s}) )&\text{ if }v(D)=1,\\[1ex]
  \displaystyle-p^{e(s-1)+2s-4+v(f)(3-2s)}\chi(-2pf)\tilde a_p&\text{ if }v(D)=0, v(f)=e+1\\[1ex]
  \displaystyle\chi(-2rp)\frac{p^{e(2-s)}}{1-p^{3-2s}}(p^{3-2s}-p^{-1}-p^{(v(f)-e-1)(3-2s)}(1-p^{-1}))&\\[2ex]
    \displaystyle\qquad-p^{e(s-1)+2s-4+v(f)(3-2s)}\chi(-2pf)\tilde a_p&\text{ if }v(D)=0, v(f)>e+1,
 \end{cases}
\end{align}
where
\begin{equation}\label{rank2badeq104}
 \tilde a_p=\begin{cases}
        p-\#\left\{(x,y)\in\F_p\times\F_p\mid y^2=x((x-b)^2-D)\right\}&\text{if }v(f)=v(r),\\
        \chi(r/f)&\text{if }v(f)>v(r),
       \end{cases}
\end{equation}
with $b\in\Z$ such that $\frac rf+p\Z_p=b+p\Z_p$.

\subsection{The calculation of \texorpdfstring{$K(s,T,\chi)$}{} for \texorpdfstring{$r=0$}{}}
In this section we assume $p$ is odd and $T=\mat{n}{0}{0}{m}$ with $n,m\neq0$. The quantity in \eqref{rank2badeq4} becomes
\begin{equation}\label{calcofKr=0eq0}
 K(s,T,\chi)=\sum_{j=1-n_p}^\infty\;p^{j(2-s)}\int\limits_{S(j+1,n_p)}\chi\left(n\mu^{-1}+m\mu p^{-2n_p}\right)\,d\mu.
\end{equation}
The following lemma will be useful to carry out certain $p$-adic integrations.
\begin{lemma}
    Let $p$ be an odd prime and $i\geq 1$.  Let $a,b\in p\Z_p$ and $u\in\Z_p^\times$.  Then either one of the maps
    \begin{align}\label{maplemma1}
         \mu\mapsto \mu^{-1}+a\mu, \qquad 
        \mu\mapsto\frac{a\mu^2+2u\mu}{a\mu+u},
    \end{align}from 
    $\Z_p^\times\rightarrow \Z_p^\times$ descends to a bijection $\Z_p^\times/(1+p^i\Z_p)\rightarrow \Z_p^\times/(1+p^i\Z_p)$.
    
\end{lemma}

\begin{proof}
    This proof is an exercise and makes use of Hensel's Lemma for surjectivity.
\end{proof}

\begin{proposition}\label{calcofKr=0lemma}
    Let $p$ be an odd prime.  Suppose $r=0$ and $n,m\neq0$. Assume $\chi$ is quadratic. Let $-4mn=Df^2$ as in Lemma~\ref{funddisclemma}. Then 
    \begin{equation}\label{calcofKr=0lemmaeq2}
        K(s,T,\chi)=
        \begin{cases}
         -a_p p^{e(2-s)-1}\chi(2pf)&\text{if }\chi(-1)=1\text{ and }v(n)=v(m)-2n_p,\\
         0&\text{otherwise},
        \end{cases}
    \end{equation}
    where $e=v(n)$ and $a_p$ is as in \eqref{apdefeq}.
\end{proposition}
\begin{proof}
It is immediate from \eqref{calcofKr=0eq0} that $K(s,T,\chi)=0$ if $\chi(-1)=-1$. Hence we assume $\chi(-1)=1$ for the rest of the proof.

Suppose that $v(n)<v(m)-2n_p$. Then the sum collapses to $j=v(n)$, and for $i\geq n_p$,
\begin{align}\label{calcofKr=0eq1}
 K(s,T,\chi)&=p^{v(n)(2-s)}\chi(n)\int\limits_{\Z_p^\times}\chi(\mu^{-1}+\tfrac{m}{n}p^{-2n_p}\mu)\,d\mu\nonumber\\
    &=p^{v(n)(2-s)}\chi(n)\sum\limits_{x\in\Z_p^\times/(1+p^i\Z_p)}\chi(x^{-1}+\tfrac{m}{n}p^{-2n_p}x)p^{-i}\nonumber\\
    &\stackrel{\eqref{maplemma1}}{=}p^{v(n)(2-s)}\chi(n)\sum\limits_{x\in\Z_p^\times/(1+p^i\Z_p)}\chi(x)p^{-i}=p^{v(n)(2-s)}\chi(n)\int\limits_{\Z_p^\times}\chi(x)dx=0.
\end{align}
Similarly we see that $K(s,T,\chi)=0$ for $v(n)>v(m)-2n_p$. From now on we assume that $v(n)=v(m)-2n_p$. Then
    \begin{equation}\label{calcofKr=0lemmaeq4}
        \Z_p^\times \ni \frac{m}{n}p^{-2n_p}=\frac{4mn}{4n^2p^{2n_p}}=-\frac{Df^2}{4n^2p^{2n_p}}=-\bigg(\frac{f}{2np^{n_p}}\bigg)^2 D.
    \end{equation}
It follows that $v_p(D)=0$ and $\frac{f}{2np^{n_p}}\in\Z_p^\times$.  We claim the sum in \eqref{calcofKr=0eq0} still collapses to $j=v(n)$, i.e.,
\begin{equation}\label{calcofKr=0eq2}
    K(s,T,\chi)=p^{v(n)(2-s)}\int\limits_{S(v(n)+1,n_p)}\chi(n\mu^{-1}+m\mu p^{-2n_p})d\mu.
\end{equation} To see this, consider the set 
\begin{equation}\label{calcofKr=0eq3}
    S(j+1,n_p)=\{\mu\in\Z_p^\times\mid v(\tfrac{n}{m}p^{2n_p}+\mu^2)=j-v(n)  \}.
\end{equation}
If $-\tfrac{n}{m}p^{2n_p}\not\in\Z_p^{\times 2}$, then $v(\tfrac{n}{m}p^{2n_p}+\mu^2)=0$. Hence, $S(j+1,n_p)=\emptyset$ if $j\neq v(n)$. If $-\tfrac{n}{m}p^{2n_p}=u^2$ for some $u\in\Z_p^{\times}$, then $S(j+1,n_p)=\{\mu\in\Z_p^\times\mid v(\mu-u)+v(\mu+u)=j-v(n)\}$. 

If $j<v(n)$, then $S(j+1,n_p)=\emptyset$.  If $j>v(n)$, 
\begin{align}\label{calcofKr=0eq6}
    S(j+1,n_p)&=\{\mu\in\Z_p^\times\mid v(\mu-u)=j-v(n)\}\sqcup \{\mu\in\Z_p^\times\mid v(\mu+u)=j-v(n)\}\nonumber\\
    &=(u+p^{j-v(n)}\Z_p^\times)\sqcup(-u+p^{j-v(n)}\Z_p^\times).
\end{align} Then
\begin{align}\label{calcofKr=0eq8}
    \int\limits_{S(j+1,n_p)}\chi(n\mu^{-1}+m\mu p^{-2n_p})d\mu&=\int\limits_{u+p^{j-v(n)}\Z_p^\times}+\int\limits_{-u+p^{j-v(n)}\Z_p^\times}\nonumber\\
%    &=2\int\limits_{u+p^{j-v(n)}\Z_p^\times}\chi(n\mu^{-1}+m\mu p^{-2n_p})d\mu\nonumber\\
    &=2\int\limits_{p^{j-v(n)}\Z_p^\times}\chi\bigg(\frac{n}{\mu+u}+m(\mu+u) p^{-2n_p}\bigg)d\mu\nonumber\\
%    &=2p^{v(n)-j}\int\limits_{\Z_p^\times}\chi\bigg(\frac{n}{\mu p^{j-v(n)}+u}+m(\mu p^{j-v(n)}+u) p^{-2n_p}\bigg)d\mu\nonumber\\
    &=2p^{v(n)-j}\chi(mp^{j-v(n)-2n_p})\int\limits_{\Z_p^\times}\chi\bigg(\frac{\mu^2 p^{j-v(n)}+2\mu u}{\mu p^{j-v(n)}+u}\bigg)d\mu\nonumber\\
    &\overset{\eqref{maplemma1}}{=}2p^{v(n)-j}\chi(mp^{j-v(n)-2n_p})\int\limits_{\Z_p^\times}\chi(\mu)d\mu=0.
\end{align}
This proves our claim \eqref{calcofKr=0eq2}. We now calculate
    \begin{align}\label{calcofKr=0lemmaeq7}
        K(s,T,\chi)&=p^{e(2-s)}\int\limits_{ \{\mu\in\Z_p^\times\mid v(n\mu^{-1}+m\mu p^{-2n_p})=e\} }\chi(n\mu^{-1}+m\mu p^{-2n_p})\,d\mu\nonumber\\
        &=p^{e(2-s)}\chi(n)\int\limits_{ \{\mu\in\Z_p^\times\mid v(\mu^{-1}-(\tfrac{f}{2np^{n_p}})^2D\mu)=0\} }\chi\bigg(\mu^{-1}-\bigg(\frac{f}{2np^{n_p}}\bigg)^2D\mu\bigg)\,d\mu\nonumber\\
        &\overset{\mu\mapsto\frac{2np^{n_p}}{f}\mu}{=}p^{e(2-s)}\chi\bigg(\frac{f}{2p^{n_p}}\bigg)\int\limits_{ \{\mu\in\Z_p^\times\mid v(\mu^{-1}-D\mu)=0\} }\chi(\mu^{-1}-D\mu)\,d\mu\nonumber\\
        &=p^{e(2-s)}\chi(2pf)\int\limits_{ \{\mu\in\Z_p^\times\mid v(\mu^3-D\mu)=0\} }\chi(\mu^3-D\mu)\,d\mu\nonumber\\
%        &=p^{e(2-s)}\chi(2pf)\sum\limits_{x\in(\Z/p\Z)^\times}\:\int\limits_{ \{\mu\in x(1+p\Z_p)\mid v(x^3-Dx)=0\} }\chi(x^3-Dx)\,d\mu\nonumber\\
        &=p^{e(2-s)-1}\chi(2pf)\sum\limits_{x\in\Z/p\Z}\bigg(\frac{x^3-Dx}{p}\bigg).
    \end{align}
It is easy to see that
\begin{equation}\label{calcofKr=0lemmaeq6}
        \#\{(x,y)\in\F_p\times \F_p\mid y^2=x^3-Dx\}=\sum_{x=0}^{p-1}\bigg(\frac{x^3-Dx}{p}\bigg)+p.
\end{equation}
This concludes the proof.
\end{proof}

Note that the cubic curve $y^2=x^3-Dx$ has discriminant $64D^3$, and hence is an elliptic curve over $\mathbb{F}_p$. See Theorem~4.23 of \cite{Washington2008} for a determination of the numbers~$a_p$.%
\bibliography{my}{}
\bibliographystyle{plain}

\end{document}